\documentclass[draft,reqno,11pt]{amsart}

 \usepackage{amssymb,amsmath, geometry}
 \usepackage{color}
\allowdisplaybreaks[4]
\numberwithin{equation}{section}
\newtheorem{theorem}{Theorem}[section]
\newtheorem{definition}[theorem]{Definition}

\newtheorem{proposition}[theorem]{Proposition}
\newtheorem{lemma}[theorem]{Lemma}

\DeclareMathOperator{\R}{R}

\DeclareMathOperator{\Weyl}{Weyl}

 \newcommand{\<}{\left\langle}
\renewcommand{\>}{\right\rangle}
 \renewcommand{\(}{\left(}
\renewcommand{\)}{\right)}
\renewcommand{\[}{\left[}
\renewcommand{\]}{\right]}
\newcommand{\eps}{\varepsilon}
\newcommand{\e}{\varepsilon}
\begin{document}

\title[Towering phenomena for the Yamabe equation on symmetric manifolds]{Towering phenomena for the Yamabe equation on symmetric manifolds}

\author{Filippo Morabito}
\address{Filippo Morabito, KAIST 291 Daehak-ro Yuseong-gu, Daejeon, Republic of  Korea,   305-701 }
\email{morabitf@gmail.com}

\author{Angela Pistoia}
\address{Angela Pistoia, Dipartimento di Scienze di Base e Applicate per l'Ingegneria, Sapienza Universit\`a di Roma,
via Antonio Scarpa 16, 00161 Roma, Italy}
\email{angela.pistoia@uniroma1.it}

\author{Giusi Vaira}
\address{Giusi Vaira, Dipartimento di Scienze di Base e Applicate per l'Ingegneria, Sapienza Universit\`a di Roma,
via Antonio Scarpa 16, 00161 Roma, Italy}
\email{vaira.giusi@gmail.com}

\maketitle
 \begin{abstract}

Let $\(M,g\)$ be a compact smooth connected   Riemannian manifold (without boundary) of dimension $N\ge7$. 
Assume $M$ is symmetric with respect to a point $\xi_0$ with non-vanishing  Weyl's tensor. We consider
   the  linear perturbation of the Yamabe problem
$$(P_\epsilon)\qquad-\mathcal L_g u+\epsilon u=u^{N+2\over N-2}\ \hbox{in}\ (M,g) .$$ We prove that
for any $k\in \mathbb N$,  there exists $\eps_k>0$ such that for all $\eps\in (0, \eps_k)$ the problem $(P_\epsilon)$ has a symmetric  solution   $u_\eps ,$  which looks like the superposition of  $k$ positive bubbles centered at the  point $\xi_0$ as $\eps\to 0$.   In particular, $\xi_0$ is a {\em towering} blow-up point. 
 \end{abstract}


\section{Introduction}
Let $\(M,g\)$ be a compact smooth connected   Riemannian manifold (without boundary) of dimension $N\ge3$.
 The  Yamabe conjecture claims that {\it the conformal class of the metric $g$ contains a metric with constant scalar curvature}.
  From a PDE's point of view, it turns to be equivalent to state that {\it the critical problem
  \begin{equation}\label{ya}
\mathcal L_g u+\kappa u^{N+2\over N-2}=0\ \hbox{in}\ M,\end{equation}
has a positive solution for some constant $\kappa.$}
   Here $\mathcal L_gu:= \Delta_g u-{N-2\over 4(N-1)}R_gu$ is the conformal laplacian, $\Delta_g$ is the Laplace-Beltrami operator and   $R_g$ is the scalar curvature of the manifold.
Indeed,   the scalar curvature of the metric $\tilde g= u^{4\over N-2} g$ (which belongs to the conformal class  of $g$) is equal to the constant ${4(N-1)\over N-2}\kappa.$\\

The Yamabe conjecture has been proved through the works of   Yamabe \cite{y}, Aubin \cite{a}, Trudinger \cite{t} and Schoen \cite{s1}.  Different proofs in low dimension, i.e. $N=3,4,5$
 and in the case
$(M, g)$
is locally conformally flat are given by Bahri and Brezis \cite{bb} and Bahri   \cite{ba}.\\

Once the question of existence was settled, it is natural to address the problem of uniqueness of the solution. Actually
the solution is unique  in the case of negative scalar curvature and it is unique (up to a constant factor)  in the case of zero scalar curvature, while  in the case of positive scalar curvature the uniqueness does not hold true anymore as it was showed by 
   Schoen in \cite{s2} and Pollack  in \cite{p} where examples of manifolds with 
  a large number of high energy solutions with high  {Morse} 
index were built. That is why a relevant part of the 
 the research work has been devoted to understand the structure of the set of the solutions.\\ 
 
In particular, Schoen  in his topics course at Stanford (see \cite{s3})  conjectured that   the set of solutions (in the positive case) is compact. 
It is important to note that
in the case of the round sphere $(\mathbb S^{N},g_0)$ the set of solutions is not compact as proved by  Obata in \cite{o}. 
Schoen's conjecture   turns out to be true when the dimension of the manifold satisfies $3\le N\le 24$ as it was shown by Khuri, Marques and Schoen \cite{kms})
(previous  results were obtained  by Schoen \cite{s4}, Schoen and  Zhang \cite{sz}, Li and  Zhu \cite{lz}, Li and  Zhang \cite{lzha}, Marques \cite{m} and Druet \cite{d1}),
while it is false when $N\ge25$ thanks to the examples built by  Brendle \cite{b} and Brendle and Marques \cite{bm}.
\\

From a PDE's point of view, the compactness issue is equivalent to establishing a priori estimates for solutions to the equation \eqref{ya}.
Therefore, 
to study the compactness of solutions to the Yamabe equation, it is crucial to establish sharp
estimates of blowing-up solutions  and in particular to find out their right asymptotic profile   near a blow-up point. 
 In particular, when the compactness holds all the possible blow-up points of a  sequence of solutions to  \eqref{ya}  must be {\em isolated} and {\em simple,} i.e. 
  around each blow-up point $\xi_0$ the solution can be approximated by a so called {\it standard bubble }
 $$u_n(x)\sim \alpha_N{ \mu_n^  {N-2\over 2}\over \(\mu_n^2+\(d_g(x,\xi_n\))^2 \)^{N-2\over 2}} \ \hbox{
 for some $\xi_n\to\xi_0$ and $\mu_n\to0.$}$$
Let us be more precise. 
 Let $u_n$ be a sequence of solutions to problem \eqref{ya}. We say that $u_n$  blows-up at a point $\xi_0\in M$ if    there exists $\xi_n\in M$ such that
 $\xi_n\to \xi_0\ \hbox{and}\ u_n(\xi_n)\to+\infty.$   $\xi_0$ is said to be a {\em blow-up point} for $u_n.$ 
 Blow-up points can be classified according to the definitions introduced by  Schoen  in \cite{s3}.
$\xi_0\in M$ is an {\em isolated blow-up point} for $u_n$ if there exists $\xi_n\in M$ such that
  $\xi_n$ is a local maximum of $u_n,$ 
 $\xi_n\to \xi_0,$ 
   $u_n(\xi_n)\to+\infty$ and
   there exist $c>0$ and $R>0$ such that $$0<u_n(x)\le c{1\over \(d_g(x,\xi_n)\)^{ {N-2\over2}}}\ \hbox{for any}\ x\in B(\xi_0,R).$$  
 Moreover, $\xi_0\in M$ is an  {\em isolated and simple blow-up point} for $u_n$ if  the function
    $$\hat u_n(r):=r^{N-2\over 2}{1\over |\partial B(\xi_n,r)|_g}\int\limits_{\partial B(\xi_n,r)}u_n d\sigma_g$$
    has a exactly one critical point in $ (0,R).$  
 \\
 
  Motivated by the previous consideration,  it is natural to ask if  
the  {\it linear perturbation} of the Yamabe problem
  \begin{equation}\label{Eq1}-\mathcal L_g u+\epsilon u=u^{N+2\over N-2},\ u>0,\ \hbox{in}\ (M,g) \end{equation}
 \begin{itemize}
\item[(i)]
 has    solutions  with one or more blow-up points    as $\epsilon\to0,$ 
 \item[(ii)] has blowing-up solutions   whose     blow-up  points are  not isolated, i.e.   {\em clustering} blow-up points,
 \item[(iii)] has blowing-up solutions   whose     blow-up  points are not  neither isolated nor simple, i.e.   {\em towering} blow-up points.
\end{itemize}

 Here we assume that  the first eigenvalue of   $-\mathcal L_g $ is positive and $\epsilon$ is a small parameter.   
  Concerning question (i), 
 Druet in \cite{d1} proved that equation \eqref{Eq1} does not have any blowing-up solution when    $\epsilon<0$ and   $N=3,4,5$
   (except when the manifold is conformally equivalent to the round sphere). It is completely open the case when the dimension is $N\ge6.$
   The situation is completely different when 
   $\epsilon>0.$  Indeed, if  $N=3$  no blowing-up solutions exist  as proved by  Li-Zhu  \cite{lz}, while if $m\ge4$  blowing-up solutions do exist  as shown by Esposito, Pistoia and Vetois in \cite{epv}. In particular, if the dimension $N\ge6$ and the manifold is not locally conformally flat, Esposito, Pistoia and Vetois built solutions which blow-up at non-vanishing stable critical points $\xi_0$ of the Weyl's tensor, i.e. $|\Weyl_g(\xi_0)|_g\not=0.$ 
   Recently, Pistoia and Vaira in \cite{pv} showed that 
    $\xi_0$ is a  {\it clustering}  blow-up point as soon as it is a  non-degenerate minimum point of the Weyl's tensor with non-vanishing Weyl's tensor.  This result  
   gives a positive answer to question (ii). We also quote the paper \cite{rv}, where  Robert and V\'etois     built solutions having clustering blow-up points for a special class of perturbed Yamabe type equation.  \\
   
   In this paper, we address question (iii) and we prove that, under some symmetry assumptions,  it is possible to build solutions to equation \eqref{Eq1} with {\em towering} blow-up points.
   More precisely, our result reads as follows.
  
  \begin{theorem}\label{torri}
 {Assume $(M,g)$   is symmetric with respect to a point $\xi_0$ and $|{\rm Weyl}_g(\xi_0)|_g\neq 0$}
Assume $N\geq 7.$ For any $k\in \mathbb N$,  there exists $\eps_k>0$ such that for all $\eps\in (0, \eps_k)$ the problem \eqref{Eq1} has a symmetric  solution   $u_\eps ,$  which looks like the superposition of  $k$ positive bubbles centered at the  point $\xi_0$ as $\eps\to 0$. In particular, $\xi_0$ is a towering blow-up point. 
\end{theorem}

This result is new and it is in sharp contrast with what happens in the euclidean case. Indeed, let us consider the classical 
 Brezis-Nirenberg problem \cite{bn}
\begin{equation}\label{brni}
\left\{\begin{aligned}&-\Delta u+\epsilon u=u^{N+2\over N-2}\ &&
  \hbox{in}\ \Omega,\\ & u>0 \  && \hbox{in}\ \Omega, \\ &u=0 \  && \hbox{on}\ \partial\Omega\end{aligned}\right.\end{equation}
 where $\Omega\subset\mathbb R^N,$ $N\ge 3$  is   an open and bounded smooth domain.
 It is well  known that it possesses blowing-up solutions when $\epsilon<0$  is small enough and $N\ge4$ (see Han \cite{h}, Rey \cite{re1,re2} and Musso and Pistoia \cite{mp}).
 Actually,  
   all the possible  blow-up points of solutions to \eqref{brni}, when  $\epsilon<0$ and small enough, are isolated and simple, namely {\em clustering} and {\em towering} blow-up points are forbidden, as it was showed by Cerqueti
   in 	\cite{c} using the ideas of Li \cite{li2}.\\

The proof of our result  relies on a delicate finite  dimensional  Ljapunov-Schmidt reduction. As usual, we need to find a {\em good}   approximation of the solution and this is carried out in Section \ref{2}. The second step consists in    finding the remainder term and here a lot of technicalities are required because we need to split the error term into the sum of remainder terms of   different orders. Finally,  we estimate the reduced energy and again we need to be extreme careful because the leading terms appears at different orders.
 All the  proofs  of the results are postponed to the Appendix \ref{appA}, while the main steps of the reduction and the proof of Theorem \ref{torri} are given in Section \ref {3}. Section \ref{1} is devoted to exhibit examples of symmetric manifolds with non-vanishing Weyl's tensor.
 \\
 
 Finally, we conjecture that Theorem \ref{torri} is  true even if we drop the symmetry assumption provided that $\xi_0$ is a non-degenerate critical point of Weyl's tensor with non-vanishing Weyl's tensor.
 The fact that the manifold is symmetric with respect to the point $\xi_0$ simplifies considerably the proof. Indeed, we are lead to build solutions which are symmetric with respect to  the point $\xi_0,$ so that in the reduction argument we only need to take care of the concentration parameters (all the bubbles are centered at the same point $\xi_0$). We point out that our proof cannot be adapted to the general case because 
the presence of   different points where the bubbles are centered would not  allow to split the error into the sum of terms with the required properties  (in particular,   property (i) of Proposition \ref{phij}  would not be true anymore).

\section{Examples of compact symmetric manifolds with non-vanishing Weyl tensor} \label{1}

\subsection{Riemannian manifolds which are symmetric with respect to a point}

We recall that if $M$ is a compact Riemannian manifold then it is complete.
Consequently for any $p \in M,$ the exponential map $exp_{p}$ is defined on the entire tangent space $T_pM$
and any geodesic curve is defined on $R.$	 Furthermore, for any point $q \in M,$ the distance 
of $q$ to $p$ equals the
length of a piece of the unique geodesic curve joining $p$ and $q.$

\begin{definition}
\label{symmetry.1}
A Riemannian manifold $M$ is symmetric with respect to a point $p$
if there exists an isometry $H:M \to M,$ such that 
$H(p)=p$ and  $dH_p : T_pM \to T_pM $ satisfies $dH_p=-id_{T_pM}.$
\end{definition}

We observe that a geodesic curve $\gamma:R \to M,$ with $\gamma(0)=p$ and initial velocity 
vector $v\in T_p M,$ 
can be written as $\gamma(t)=exp_p(tv).$ If $H$ is an isometry, then it 
always holds true that $$H(\gamma(t))=H(exp_p(tv))=exp_{H(p)}(dH_p(tv)).$$
If in addition $M$ is symmetric with respect to $p$ and $H$ satisfies the conditions of previous definition, then
\begin{equation}
\label{reverse}
H(\gamma(t))=exp_p(-tv)=\gamma(-t).
\end{equation}

Consequently, an equivalent definition is the following. 

\begin{definition}
\label{symmetry.2}
$M$ is symmetric with respect to a point $p \in M$ if there exists an isometry
$H:M \to M,$  such that $H(\gamma(t))=\gamma(-t)$ for any 
 geodesic curve $\gamma:R \to M$ such that $\gamma(0)=p.$
In other terms the isometry $H$ reverses the geodesic curves passing by 
the point $p.$
\end{definition}

If we set $t=1$ in \eqref{reverse} then we get that the image of $exp_p(v)$ under the action of $H$ is $H(exp_p(v))=exp_p(-v),$ 
for any $v \in T_pM.$ 
Since any isometry preserves the length of curves and  $M$ is complete, then $d_g(p,exp_p(v))=d_g(p,exp_p(-v)),$
where $d_g$ denotes the distance with respect to the metric $g.$

An example of compact manifold which is symmetric with respect to a point is the unit
sphere $S^n=\{(x_1,\ldots,x_{n+1}) \in \R^{n+1}, \sum_1^{n+1} x_i^2=1\}$ equipped with the standard metric.
$S^n$ is symmetric with respect to any point $p \in S^n.$ 
We show that holds true in the case where $p$ coincides with the south pole $S$, the point having 
coordinates $(0,\ldots,0,-1).$ 

We define a map $H:R^{n+1}\to R^{n+1}$ setting $H(x_1,\ldots,x_{n+1})=(-x_1,-x_2,\ldots,-x_n,x_{n+1}).$ 
It is immediate to check this map is an isometry of $R^{n+1}.$ 
Consequently the restriction $h$ of this map to the sphere is an isometry of $S^n$ as well and it fixes the south pole $S.$
Furthermore its differential satisfies $dh_S=-id_{T_S S^n}.$

\subsection{Riemannian manifolds with non-vanishing Weyl tensor}


It is known that a $n$-dimensional Riemannian manifold is locally conformally 
flat if and only if the Cotton tensor vanishes identically in the case $n=3$
and if and only if the Weyl tensor vanishes identically in the case $n \geq 4.$

Any space having constant sectional curvature is known to be locally conformally flat:
 $S^n(c)$ (the sphere of radius $\frac1{\sqrt{c}}$)    and the hyperbolic space $H^n(-c),$ $c>0,$ 
 have sectional curvature equal to $c,-c$    respectively.

A useful procedure to produce examples of Riemannian 
manifolds which are not locally conformally flat consists 
in considering the product or more generally the warped
 product of (eventually locally conformally flat) manifolds.

We start by recalling the definition of warped product of
two Riemannian manifolds $(B,g_B)$ and $(F,g_F).$

The {\it warped product }  $B \times_f F$ is the Riemannian manifold
$(B \times F,g),$ where $g =g_B \otimes f^2 g_F$
and $f : B \to \R$ is a positive function called warping function.

Theorem 1 in \cite{bgv1} provides the classification of the  warped products which are
locally conformally flat Riemannian manifolds. 

\begin{theorem}
 \label{warped}
We set $M:=B \times_f F.$

\begin{enumerate}
	\item If $dim(B)=1,$ then $M$ is locally conformally flat 
	if and only if $(F,g_F)$ is a space of constant sectional curvature.
	
	\item If $dim(B)>1,$ $dim(F)>1,$ then $M$ is locally conformally flat 
	if and only if the two following conditions are satisfied:
	
	\begin{itemize}
		\item $(F,g_F)$ is a space of constant curvature;
		
		\item the warping function $f$ defines a conformal deformation
		on $B,$ such that $(B, \frac1{f^2} g_B)$ has constant
		sectional curvature
		equal to $-c_F.$
	
	\end{itemize}
	
	\item If $dim(F)=1,$ then $M$ is locally conformally flat 
	if and only if the warping function $f$ defines a conformal deformation
		on $B,$ such that $(B, \frac1{f^2} g_B)$ has constant sectional curvature.
\end{enumerate}

\end{theorem}

If in the definition of warped product we allow the warping function $f$
to be defined on the whole set $B \times F,$ then we get the definition of
{\it twisted product} of  $(B,g_B)$ and $(F,g_F).$ 

A necessary condition for a twisted product to be locally conformally flat 
is provided by the following theorem (Theorem 6 in \cite{bgv1}) :

\begin{theorem}
Suppose $dim(B)>1,$ $dim(F) >1. $
If the twisted product $B \times_f F$ is a locally conformally flat manifold then
it can be expressed as warped product.
\end{theorem}
 
It is easy to check that a twisted product can be regarded as a warped product 
if and only if $f$ is the product of two functions $f_1,f_2,$ the first being defined 
on $B,$ the second being defined on $F.$  
If such a condition is not satisfied the twisted product is not locally conformally flat. 

The third class of manifolds we consider is the one which consists in 
multiply warped products.

Given the Riemannian manifolds 
$(B,g_B),$ $(F_i,g_{F_i}),$ with $i\in \{1,\ldots,k\},$ $k\geq 2,$
and $g_R$ the euclidean metric,
then their multiply warped product  $B \times_{f_1} F_1 \times_{f_2} F_2 \times \cdots \times_{f_k} F_k$ is the Riemannian 
manifold $(B \times F_1 \times F_2 \times \cdots \times F_k ,g),$
with $g=g_R \otimes f^2_1 g_{F_1} \otimes \cdots \otimes f^2_k g_{F_k}$
and $f_i : F_i \to \R$ is a positive function.

We will also assume that $f_i$ is non-constant and the spaces $F_i$ are different.

Lemma 3.1, Remark 3.5 (see also the considerations done at page 210) in \cite{bgv2} 
provide some necessary conditions for a  multiply warped product to be a locally conformally flat
 manifold. We mention only the following result.

\begin{lemma}
\label{necessary.conditions}
If $(B \times F_1 \times F_2 \times \cdots \times F_k ,g)$ is locally conformally flat then $k\leq b+2,$
where $b=dim(B),$ $(B,g_B)$ is locally conformally flat and $(F_i,g_{F_i})$ have constant sectional curvature (provided $dim(F_i)\geq 2 $).
\end{lemma}

\subsection{Compact symmetric manifolds without boundary and non-vanishing Weyl tensor}

In this subsection we explain how to use warped products in order 
to produce examples of compact Riemannian manifolds of dimension at least $4,$ without boundary,
which are symmetric with respect to a point and have non-vanishing Weyl tensor.

First we explain under which conditions a warped product $M \times_f N$ is symmetric with
 respect to a point if $M,N$ are. Let $h_M,h_N$ denote the isometries which satisfy the conditions of Definition
\ref{symmetry.1}. 

\begin{lemma}
\label{symmetry.product}
We suppose $(M,g)$ is a Riemannian manifold which is symmetric with respect to $p \in M,$
and $(N,\tilde g)$ is a Riemannian manifold which is symmetric with respect to $q \in N,$
Then the warped product $(M \times N, G),$ with $G=g \otimes f^2 \tilde g,$ is 
 symmetric with respect to the point $(p,q) \in M \times N,$
if the warping function $f:M\to \R^+$ satisfies $f \circ h_M=f.$   In other terms
$f$ is invariant under the action of the isometry $h_M.$
\end{lemma}

\begin{proof}
The map $h:M \times N \to M\times N  ,$ defined as $\pi_M \circ h=h_M,$ $\pi_N \circ h=h_N,$ where
$\pi_M : M\times N \to M,$  $\pi_N : M\times N \to N,$  denotes the projections, is an isometry of $M \times N.$

In order to show that, we assume that 
\begin{itemize}
	\item $r \in M,$ $s \in N.$
\item $V,W \in T_{(r,s)}M \times N \cong T_{r}M \oplus T_{s}N,$  
\item $V_1,V_2$  are the projections of $V$ on $T_{r}M$ and $T_{s}N,$ 
\item $W_1,W_2$ are the projections of $W$ on $T_{r}M$ and $T_{s}N.$
\end{itemize}
$h$ is an isometry if it is a diffeomorphism (the proof of this is immediate) and
$$G_{h(r,s)}(dh_{(r,s)}(V),dh_{(r,s)}(W))=G_{(r,s)}(V,W).$$
By definition of the metric $G,$ the right hand side equals
$$g_{h_M(r)} ((dh_M)_r(V_1),(dh_M)_r (W_1))+
[f(h_M(r))]^2 \tilde g_{h_N(s)} ((dh_N)_s(V_2),(dh_N)_s (W_2)).$$
Using the fact that $h_M$ and $h_N$ are isometries and $f\circ h_M=f$, we can write that as:
$$ g_r(V_1,W_1)+ f^2(r) \tilde g_s(V_2,W_2)=G_{(r,s)}(V,W). $$

It remains to show that $dh_{(p,q)}$ coincides with the antipodal map on 
$T_{(p,q)}M \times N \cong T_{p}M \oplus T_{q}N.$
This follows from:
$dh_{(p,q)}=((dh_M)_{p},(dh_N)_{q})=(-id_{T_{p}M},-id_{T_{q}N})=-id_{T_{p}M \oplus T_{q}N}.$  
\end{proof}

The $n$-spheres $S^n$ are examples of  compact  manifolds which are symmetric with respect to a point,
but their Weyl tensor vanishes identically because they are locally conformally flat manifolds.
 We can obtain manifolds which are not locally conformally flat
if we take the product of at least two spheres.

The Riemannian manifold $(P,g_P)=(S^{n} \times S^{m},g_{S^{n}}\otimes f^2 g_{S^{m}})$  with $n,m \geq 2$ 
 is not locally conformally flat for any choice of the warping function $f$ which does not satisfy the hypotheses of Theorem \ref{warped}, part (2). 
In particular when $f$ is a constant function. 

Such manifolds are also symmetric with respect to a point  provided $f$ satisfies the condition
of Lemma \ref{symmetry.product}.

Alternatively, we can consider the twisted product of two spheres. 

Products of an higher number of spheres can be shown to be not locally conformally flat, writing it as a product of 
two manifolds and using induction. For example  
$ S^{l} \times S^{n} \times S^{m}$   equipped with the metric $g_{S^l}\otimes g_{S^n}\otimes g_{S^m}$ is not locally conformally flat,
because we can write it as product of $S_l$ and the manifold $P$ constructed above, with $f \equiv 1.$
Now we can use again Theorem \ref{warped}, because $(P,g_P)$ has non-constant sectional curvature.
The last assertion follows from the fact that if $P$ was a manifold with constant sectional curvature then it would be
locally conformally flat.

 Similarly, the multiply warped product $(M_k,g_k)=(S^{n} \times S^{m_1} \times  \cdots  S^{m_k}, g_{S^{n}}\otimes f_1^2 g_{S^{m_1}}\otimes  \cdots 
f_k^2 g_{S^{m_k}} ),$ $n,m_i \geq 2,$ $k\geq n+3,$
is not locally conformally flat for any choice of the warping functions $f_i,$ according to Lemma \ref{necessary.conditions}.

In order to study the symmetry,  we observe that by Lemma \ref{symmetry.product}
we  can show the symmetry $(M_k,g_k)$ arguing by induction.
The product $M_1$ is symmetric if $f_1 \circ h_0=f_1, $ where $h_0$ is the isometry of $S^n.$
The product $M_{i+1}$ is symmetric if $f_{i+1}\circ h_i=f_{i+1},$ where $h_i$ is the isometry of $M_i,$
with $i \in \{1,\ldots,k-1\}.$

Examples with same structure are those ones we get if we replace the spheres by other compact manifolds.
Let us consider the $n$-dimensional ellipsoids, $n\geq 2$, centered at the origin, 
$$\left\{(x_1,x_2,\ldots,x_{n+1})\in \R^{n+1}, \sum_1^{n+1}
\frac{x_i^2}{a_i^2}=1\right\},$$ with $a_i>0,$ 
endowed with the metric induced by the euclidean one. 
 Note that if the semi-axis length $a_i=r$ for any $i\in \{1,\ldots,n+1\},$ then we get the $n$-sphere of radius $r.$ 

A direct computation shows that an $n$-dimensional ellipsoid in $R^{n+1}$ equipped 
with the metric induced by the euclidean one,  is not locally conformally flat if $n\geq 3$ and at least three
of the semi-axis lengths $a_i$ are different. See also the Proposition by Cartan, Schouten in \cite{k}.

Using $m$-dimensional ellipsoids, $m\geq 2,$ for which at least two of the semi-axis 
lengths are different  (this hypothesis ensures that their curvature is not constant), then, in 
view of Theorem \ref{warped}, we can 
construct warped products which are not locally conformally flat.
As in the case of product of spheres, we can show  the symmetry of these examples using the
symmetry of each ellipsoid with respect to one of its vertices.  

Of course there are plenty of other examples, the ones we presented here are relatively easy to describe.

\subsection{Examples of symmetric manifolds with nowhere vanishing Weyl tensor}
In view of previous considerations, we already know that the product of sphere is
symmetric with respect to a point. We finish the section by showing that a product of spheres is
is example of compact Riemannian manifold which has nowhere vanishing 
Weyl tensor. 

The sphere $S^m$ equipped with the standard metric enjoys the following property: 
Isometries of $S^m$ act transitively, that is for each fixed pair of distinct points
$p,q \in S^m,$ there exists an isometry $H:S^m \to S^m,$ such that $H(p)=q.$
This property is clearly inherited by any product $S^{m_1} \times S^{m_2}.$ 

Now we assume $m_1,m_2 \geq 2.$ Since the Weyl tensor is preserved by isometries,
if the Weyl tensor vanishes at $(p_1,p_2) \in S^{m_1} \times S^{m_2}$  then it vanishes also
at the point $H(p_1,p_2),$  where $H$ is an isometry from $S^{m_1} \times S^{m_2}$ 
onto itself. Since $H(p_1,p_2)$ can be chosen arbitrarily this would show that the Weyl tensor
 vanishes at each point. That says $S^{m_1} \times S^{m_2}$ would be locally conformally
flat, and that contradicts Theorem \ref{warped}, part (2), with $f\equiv 1.$

The same proof applies to multiple products of spheres.


\section{The ansatz}\label{2}

\subsection{Preliminaries}
We will assume that $M$ is symmetric with respect to a point $\xi$ with $|{\rm Weyl}_g(\xi )|_g\neq 0.$ We will also assume that $M$ has dimension $N\ge7.$\\

The main ingredient in our construction are the euclidean bubbles
\begin{equation}\label{bubble}U_{\mu,y}(x) = \mu^{-{N-2\over2}} U\({  x-y \over \mu}\) ,\ x,y\in\mathbb R^N,\ \mu>0,\ \hbox{where}\
 U (x):=\displaystyle{ {\alpha_N}{1\over\(1+|x |^2\)^{N-2\over2}}}.\end{equation} 
Here $\alpha_N:={N(N-2)}^{N-2\over4}.$ They are all the solutions to the  critical equation in the Euclidean space
  \begin{equation}\label{criti} -\Delta U=U^{N+2\over N-2} \  \hbox{in}\  \mathbb R^N.\end{equation}

Let us consider the euclidean bubble  $U_{\mu,0}$ , centered around the origin (see \eqref{bubble}),
which via a geodesic normal coordinate system around the point $\xi\in M $ read as
$$\mathcal U_{\mu,\xi}(z)= U_{\mu,0} \left( {\rm exp}_{\xi}^{-1}(z)    \right)=\mu^{-{N-2\over2}}U  \left(\frac{{\rm exp}_{\xi}^{-1}(z) }{\mu } \right)\ \hbox{if $d_g(\xi,z)$ is small enough.}$$
A comparison between the conformal laplacian   with the euclidean laplacian of the bubble shows that  there is an error, which at main order looks like
 \begin{equation}\label{conf-eu} \mathcal L_g \mathcal U_{\mu,\xi}   -\Delta \mathcal U_{\mu,\xi} \sim  -\frac 13 \sum\limits_{a,b,i,j=1}^NR_{iab j}(\xi)x_a x_b\partial^2_{ij}U_{\mu,0}+\sum\limits_{i, l,k=1}^N \partial_l \Gamma^k_{ii}(\xi)x_l\partial_k U_{\mu,0}+{\beta_N R_g(\xi)}U_{\mu,0}.\end{equation}
 Here $\beta_N:={N-2\over 4(N-1)}$, $ R_{iab j}$ denotes the Riemann curvature tensor,   $\Gamma^k_{ij}$ the Christoffel's symbols
 and $R_g$ the scalar curvature.
This  easily follows by standard properties of the exponential map, which imply 
\begin{equation}\label{lb}
-\Delta_g u = -\Delta u -(g^{ij}-\delta^{ij})\partial^2_{ij}u+g^{ij}\Gamma^k_{ij}\partial_k u, 
\end{equation}
with
\begin{equation}\label{gij}
g^{ij}(x)=\delta^{ij}(x) -\frac 13 R_{iab j}(\xi)x_a x_b +O(|x|^3)\ \hbox{and}\ 
g^{ij}(x)\Gamma^k_{ij}(x)=\partial_l \Gamma^k_{ii}(\xi)x_l +O(|x|^2).
\end{equation}
 To build our solution it shall be necessary    to kill the R.H.S of \eqref{conf-eu} by adding to the bubble an higher order term $V$   whose existence has been established in \cite{ep}. To be more precise, we need
 to remind (see \cite{be}) that the all the solutions to the linear problem
\begin{equation}\label{Eqlin}
-\Delta v=pU^{p-1}v\quad\hbox{in}\ \mathbb{R}^N,\end{equation}
are   linear combinations of the functions
\begin{equation}\label{Eq14}
\psi^0\(x\)=x\cdot\nabla U(x)+{N-2\over2} U(x)
\ \hbox{and}\
 \psi^i\(x\)=\partial_i U(x), \ i=1,\dots,N.
\end{equation}
The correction term $V$ is built in the following proposition (see  Section 2.2 in \cite{ep}).
 \begin{proposition}\label{bubbleV}
 There exist  $\nu(\xi)\in\mathbb R $ and a function $V\in   \mathcal{D}^{1,2}(\mathbb{R}^N) $ solution to
\begin{align}\label{EqV1}
  -\Delta  V-f'(U) V = & -\sum\limits_{a,b,i,j=1}^N{1\over 3} R_{iabj} (\xi)x_ax_b\partial^2_{ij} U
 +\sum\limits_{i, l,k=1}^N \partial_l\Gamma^k_{ii}(\xi)x_l\partial_kU +\beta_N\R_g(\xi) U+\nu(\xi)\psi^0 \ \hbox{in}\ \mathbb{R}^N ,
\end{align}
with
$$\displaystyle\int\limits_{\mathbb{R}^N}V(x)\psi^i(x)dx=0,\  i=0,1,\dots,N $$
and
\begin{equation}\label{V3}
|V(x)|+ |x|\left|\partial_kV(x)\right|+|x|^2 \left|\partial^2_{ij}V (x)\right| =O\({1\over \(1+|x|^2\)^{N-4\over2}}\),\quad x\in\mathbb{R}^N. \end{equation}
 \end{proposition}

\subsection{The tower}
 Let $r_0$ be a positive real number less than the injectivity radius of $M$, and $\chi$ be a smooth cutoff function  such that $0\le\chi\le1$ in $\mathbb{R}$, $\chi\equiv1$ in $ \[-r_0/2,r_0/2 \]$, and $\chi\equiv0$ out $ \[-r_0,r_0 \]$. For  any positive real number $\mu_j$, we define   $W_j$ by 
\begin{equation}\label{Wj}
W_j(z):=\underbrace{\chi(d_g(z, \xi))\mu_j^{-\frac{N-2}{2}}U\left(\frac{{\rm exp}_{\xi}^{-1}(z)}{\mu_j} \right)}_{:=\mathcal U_j(z)}+\mu_j^2 \underbrace{\chi(d_g(z, \xi))\mu_j^{-\frac{N-2}{2}}V\left(\frac{{\rm exp}_{\xi}^{-1}(z)}{\mu_j} \right)}_{:=\mathcal V_j(z)}, z\in M
\end{equation}
where the functions $U$ and $V$ are defined, respectively, in \eqref{bubble} and \eqref{EqV1}.\\

 We look for {\it symmetric} solutions of \eqref{Eq1}, according to the following definition.
\begin{definition}\label{fi-simmetria}
We say that a function $u:M\to\mathbb R$ is symmetric   if
 $u(H(x))=u(x)$ for any $x\in M.$ 
 $H$ is the isometry introduced in {Definition} \eqref{symmetry.2}.
\end{definition}

 More precisely, we look for symmetric solutions of \eqref{Eq1}
 of the form 
\begin{equation}\label{sol}
u_\eps(z):= \sum_{j=1}^k W_j(z) +\Phi_{ \eps}(z)
\end{equation}
where each term $W_j$ is defined in \eqref{Wj},   and
for any $j=1, \ldots, k$ the concentration parameter $\mu_j$ satisfies
  \begin{equation}\label{muj}
\mu_{j }= d_{j } \eps^{\gamma_j} \ \hbox{with}\ d_1,\dots,d_k\in(0,+\infty)\ \hbox{and}\ 
 \gamma_j:=\left(\frac{N-2}{N-6}\right)^{j-1}-\frac 12.
\end{equation} 
We point out that the choice the concentration rate for $\mu_j$ is motivated by the fact that (see the expansion of the reduced energy in \eqref{ex-jtilde})
$$\mu_1^4\sim\eps\mu_1^2\quad \hbox{and}\quad \({\mu_j\over\mu_{j-1}}\)^{N-2\over2}\sim \eps\mu_j^2\ \hbox{for any $j\ge2.$}$$

  The remainder term $\Phi_\eps$ shall be splitted into the sum of $k$ terms of different order
\begin{equation}\label{resto}
\Phi_{ \eps}(z):=\sum_{\ell=1}^k \phi_{\ell, \eps}(z), \qquad z\in M
\end{equation}
where  each remainder term $\phi_{\ell, \eps}$ only depends on   $d_{1},\dots,d_\ell $, it is symmetric   according to Definition \ref{fi-simmetria} and
it belongs to the space $\mathcal K^\bot_{\ell}$ defined in \eqref{kelle}.\\

\subsection{Setting of the problem}

We   provide the Sobolev space $H^1_g\(M\)$ with the scalar product
\begin{equation}\label{Eq6}
\<u,v\> =\int_M \left<\nabla u,\nabla v\right>_g d\nu_g+\beta_N\int_M\R_g uvd\nu_g
\end{equation}
where $d\nu_g$ is the volume element of the manifold.   We let $\|\cdot\| $ be the norm induced by $\<\cdot,\cdot\> $. Moreover, for any function $u$ in $L^q(M)$ and for any $A\subset M$, we let $|u|_{q,A}=\left(\int_A|u|^qd\nu_g\right)^{1/q}$.

We let $\i^*:L^{\frac{2N}{N+2}}\(M\)\rightarrow H^1_g(M)$ be the adjoint operator of the embedding $\i:H^1_g\(M\)\hookrightarrow L^{{2^*}}\(M\)$, i.e. for any $w$ in $L^{\frac{2N}{N+2}}(M)$, the function $u=\i^*\(w\)$ in $H^1_g\(M\)$ is the unique solution of the equation $ -\mathcal L_g u=w$ in $M$. By the continuity of the embedding of $H^1_g\(M\)$ into $L^{2^*}\(M\)$, we get
\begin{equation}\label{Eq7}
\left\|\i^*\(w\)\right\|\le C\left |w\right |_{\frac{2N}{N +2}}
\end{equation}
for some positive constant $C$ independent of $w$.
We rewrite problem \eqref{Eq1} as
\begin{equation}\label{Eq1b}
u=\i^*\[f(u) -\eps u\],\qquad u\in{\rm H}^1_g(M)\end{equation}
where we set
  $f(u):=(u^+)^p $ with  $p={N+2\over N-2} $.

For any $j=1,\dots,k$ we set
\begin{equation}\label{Zj}
Z_j^0(z):=\chi(d_g(z, \xi))\mu_j^{-\frac{N-2}{2}}\psi^0\left(\frac{{\rm exp}_\xi^{-1}(z)}{\mu_j}\right),\ z\in M
\end{equation}
where the function $\psi^0$ is defined in \eqref{Eq14} and
for any integer $\ell=1,\dots,k$,  we define the   subspaces
\begin{equation}\label{kelle}
\begin{aligned}
 &\mathcal K_{\ell}:={\rm Span}\left\{\i^*(Z_j^0),\   j=1,\dots,\ell\right\}\\ 
 &\mathcal K^\bot_{\ell}:=\left\{\phi\in H^1_g(M)\,:\, \hbox{$\phi$ is symmetric and}\ \langle \phi, \i^*(Z_j^0)\rangle  =0,\     j=1,\dots,\ell\right\}.\\
 \end{aligned}\end{equation}
We also define $\Pi_{\ell}$ and $\Pi^\bot_{\ell}$ the projections of the Sobolev space $H^1_g(M)$ onto the respective subspaces $\mathcal K_{\ell}$ and $\mathcal K^\bot_{\ell}$. \\

In order to solve 
  equation \eqref{Eq1b}, we shall solve the system
\begin{align}\label{bif}
&\Pi^\bot_{k}\left\{u_\eps-\i^*\left[f(u_\eps)-\eps u_\eps\right]\right\}=0,\\
\label{aux}
&\Pi_{k}\left\{u_\eps-\i^*\left[f(u_\eps)-\eps u_\eps\right]\right\}=0
\end{align}
 where $u_\eps$ is given in \eqref{sol}.

\section{The Ljapunov-Schmidt procedure}\label{3}

\subsection{The remainder term: solving equation \eqref{bif}}
In order to find the remainder term $\Phi_\eps,$ we shall find functions $\phi_{j, \eps}$ for any $j=1, \ldots, k,$ which solve the following system of $k$ equations

\begin{equation}\label{sistema}
\left\{
\begin{aligned}
&\mathcal E_1+\mathcal S_1(\phi_{1, \eps})+\mathcal N_1(\phi_{1, \eps})=0\\
&\mathcal E_2+\mathcal S_2(\phi_{2, \eps})+\mathcal N_2(\phi_{1, \eps}, \phi_{2, \eps})=0\\
&\ldots\\
&\ldots\\
&\mathcal E_k+\mathcal S_k(\phi_{k, \eps})+\mathcal N_k(\phi_{1, \eps}, \ldots, \phi_{k, \eps})=0.\\
\end{aligned}
\right.
\end{equation}
The error terms $\mathcal E_\ell$ are defined by
\begin{equation}\label{R1}
\mathcal E_1:=\Pi^\bot_{1}\left\{W_1-\i^*\left[f\left( W_1\right)-\eps W_1\right]\right\}
\end{equation}
and  
\begin{equation}\label{Rj}
\begin{aligned}
\mathcal E_\ell&:=\Pi^\bot_{\ell}\left\{W_\ell-\i^*\left[f\left(\sum_{j=1}^\ell W_j\right)-f\left(\sum_{j=1}^{\ell-1}W_j\right)-\eps W_\ell\right]\right\},\ \ell\ge2.\end{aligned}
\end{equation}
The linear operators $\mathcal S_\ell$ are defined by
 for $\ell=1, \ldots, k$
\begin{equation}\label{Lj}
\mathcal S_\ell (\phi_{\ell, \eps}):=\Pi^\bot_{\ell}\left\{\phi_{\ell, \eps}-\i^*\left[f'\left(\sum_{j=1}^\ell W_j\right)\phi_{\ell, \eps}-\eps \phi_{\ell,\eps}\right]\right\}.
\end{equation}
The higher order terms $\mathcal N_\ell$ are defined by
\begin{equation}\label{N1}
\mathcal N_1(\phi_{1, \eps}):=\Pi^\bot_{1}\left\{-\i^*\left[f\left(W_1+\phi_{1, \eps}\right)-f\left(W_1\right)-f'\left(W_1\right)\phi_{1, \eps}\right]\right\}
\end{equation}
and  
\begin{equation}\label{Nj}
\begin{aligned}
\mathcal N_\ell(\phi_{1, \eps}, \ldots, \phi_{\ell, \eps}) &:=\Pi^\bot_{\ell}\left\{-\i^*\left[f\left(\sum_{j=1}^\ell (W_j+\phi_{j, \eps})\right)-f\left(\sum_{j=1}^\ell W_j\right)\right.\right.\\
&\left.\left.-f'\left(\sum_{j=1}^\ell W_j\right)\phi_{\ell, \eps}-f\left(\sum_{j=1}^{\ell-1}(W_j+\phi_{j, \eps})\right)+f\left(\sum_{j=1}^{\ell-1} W_j\right)\right]\right\},\ \ell\ge2.
\end{aligned}
\end{equation}

In order to solve system \eqref{sistema}, first of all we need to evaluate the $H^1_g(M)-$ norm of the error terms $\mathcal E_\ell $.
This is done in the following lemma whose proof is postponed in Section \ref{appA}.\\

\begin{lemma}\label{errorej}
For any $\ell=1, \ldots, k$ and for any compact subset $A_\ell\subset (0,+\infty)^{\ell} $ there exists a positive constant $C $  and $\eps_0>0$ such that for any $\eps\in(0,\eps_0)$  and for any $(d_1,\dots,d_\ell)\in A_\ell$ there holds 
\begin{equation}\label{erroreRj}
\|\mathcal E_\ell\|\leq C  \left\{\begin{aligned} &\mu_\ell^{\frac 52}+\eps \mu_\ell^2+\left(\frac{\mu_\ell}{\mu_{\ell-1}}\right)^{\frac 9 4}\qquad &\mbox{if}\,\, N=7\\ &\mu_\ell^{3}|\ln\mu_\ell|^{\frac 5 8}+\eps\mu_\ell^2+\left(\frac{\mu_\ell}{\mu_{\ell-1}}\right)^{\frac 5 2}\qquad &\mbox{if}\,\, N=8\\ &\mu_\ell^{3}+\eps\mu_\ell^2+\left(\frac{\mu_\ell}{\mu_{\ell-1}}\right)^{\frac{N+2}{4}}\qquad &\mbox{if}\,\, N\ge9,\end{aligned}\right.\end{equation}
where we agree that if $\ell=1$ the interaction term $ \frac{\mu_\ell}{\mu_{\ell-1}}$ is zero.
In particular, by the choice of $\mu_\ell$'s in \eqref{muj} we deduce
\begin{equation}\label{elle}
\begin{aligned}
&\|\mathcal E_1\|=\left\{O\(\eps^{\frac54 }\)\ \hbox{if}\ N=7,\ O\(\eps^{\frac32 }|\ln\eps|\)\ \hbox{if}\ N=8,\ O\(\eps^{\frac32 }\)\ \hbox{if}\ N\ge9\right\},\\
&\|\mathcal E_\ell\|=O\(\eps^{{p\over2}\theta_\ell}\)\ \hbox{if}\ \ell\ge2,\end{aligned}\end{equation}
where
\begin{equation}\label{tetaelle} \theta_\ell:=\(\gamma_\ell-\gamma_{\ell-1}\)\frac{N-2}{2}=1+2\gamma_\ell = 2\left(\frac{N-2}{N-6}\right)^{\ell-1},\ \ell\ge2.
 \end{equation}
\end{lemma}

Next, we need to understand the invertibility of the linear operators $\mathcal S_\ell  $.
This is done in the following lemma whose proof can be carried out as in   \cite{mp}.\\
\begin{lemma}\label{lineare}
For any $\ell=1, \ldots, k$ and for any compact subset $A_\ell\subset (0,+\infty)^{\ell} $ there exists a positive constant $C $  and $\eps_0>0$ such that for any $\eps\in(0,\eps_0)$  and for any $(d_1,\dots,d_\ell )\in A_\ell$  there holds
\begin{equation}
\|\mathcal L_\ell (\phi_{\ell})\|\geq C  \|\phi_{\ell}\|\ \hbox{for any $\phi_\ell \in \mathcal K_{\ell}^\bot$}.\end{equation}
\end{lemma}

Finally, we are able to solve system \eqref{sistema}. This is done in the following proposition, whose proof is postponed in Section \ref{appA} and relies on  a   sophisticated contraction mapping argument.\\
\begin{proposition}\label{phij}
 For any compact subset $A \subset  (0,+\infty)^{k} $ there exists a positive constant $C $ and $\eps_0$ such that for $\eps\in(0,\eps_0)$, for any $(d_1,\dots,d_k )\in A $ and   for any $\ell=1, \ldots, k$    there exists a unique function $\phi_{\ell, \,\eps}\in \mathcal K^\bot_{\ell}$  which solves the $\ell-$th equation in \eqref{sistema} such that
\begin{itemize}
\item[(i)] $\phi_{\ell, \eps}$  depends only on  $ d_1,\dots,d_ \ell$;
\item[(ii)] the map $(d_1,\dots,d_\ell )\to \phi_{\ell, \eps}(d_1,\dots,d_\ell) $ is of class $C^1$ and 
\begin{equation}\label{stimaphi1}
\begin{aligned}
&\|\phi_{1, \eps}\|=\left\{O\(\eps^{\frac54 }\)\ \hbox{if}\ N=7,\ O\(\eps^{\frac32 }|\ln\eps|\)\ \hbox{if}\ N=8,\ O\(\eps^{\frac32 }\)\ \hbox{if}\ N\ge9\right\},\\
&\|\phi_{\ell, \eps}\| =O\(\eps^{{p\over2}\theta_\ell}\)\ \hbox{if}\ \ell\ge2,\
(\theta_\ell \ \hbox{is defined in \eqref{tetaelle}});
 \end{aligned}\end{equation}
 Moreover,
 \begin{equation}\label{blabla} \|\nabla_{(d_1,\dots,d_\ell)}\phi_{\ell, \eps}\|=o(1).
 \end{equation}
  \item[(iii)] there exists $\rho>0$ such that \begin{equation}\label{stimalinfty}
\sup\limits_{d_g(x, \xi)\le \rho\mu_\ell }|\phi_{\ell, \eps}(x)|=O\( \mu_\ell^{-\frac{N-2}{2}}\).
\end{equation}  

 \end{itemize}
\end{proposition}

\subsection{The reduced problem: proof of Theorem \ref{torri}}
 Let us define the energy $J_\epsilon:H_g^1(M)\to \mathbb R$
\begin{equation}\label{energy}
J_\epsilon(u):={1\over2}\int\limits_M\(|\nabla_g u|^2+\beta_N\R_g u^2+\epsilon u^2\)d\nu_g-{1\over p+1}\int\limits_M \(u^+\)^{p+1}d\nu_g,
\end{equation}
whose critical points are solutions to the problem \eqref{Eq1}.
\\
Let us introduce the reduced energy, defined if $(d_1, \ldots, d_k ) \in (0, +\infty)^k $ by
\begin{equation}\label{red-en}
\widetilde{J}_\eps(d_1, \ldots, d_k ):=J_\eps\(\sum_{j=1}^k W_j  +\Phi_{ \eps} \),
\end{equation} where the remainder term $\Phi_{\eps}=\sum_{j=1}^k \phi_{j, \eps}$ and the $\phi_{j, \eps}$'s  are defined in Proposition \ref{phij}.

The following result allows as usual to reduce our problem to a finite dimensional one. The proof is quite involved and it is postponed in Section  \ref{appA}.\\

\begin{proposition}\label{prob-rido}
\begin{itemize}
\item[(i)] $\sum_{j=1}^k W_j  +\Phi_{ \eps}$ is a solution to \eqref{Eq1} if and only if $(d_1, \ldots, d_k ) \in (0, +\infty)^k $  is a critical point of the reduced energy \eqref{red-en}
\item[(ii)] The following expansion holds true
\begin{equation}\label{ex-jtilde}\begin{aligned}
\widetilde{J}_\eps(d_1, \ldots, d_k ) :=  D_N&+\eps^2 \left[-A_N |{\rm Weyl}_g(\xi)|^2_g d_1^4+B_N d_1^2+\Upsilon_1\right] \\ &+\sum_{\ell=2}^k \eps^{\theta_\ell}\left[-C_N \left(\frac{d_\ell}{d_{\ell-1}}\right)^{\frac{N-2}{2}} +B_N d_\ell^2 +\Upsilon_\ell\right]\\
\end{aligned}\end{equation}
as $\eps\to 0$,    uniformly with respect to $(d_1, \ldots, d_k ) $ in compact subsets of $(0, +\infty)^k   .$
 Here  $\theta_\ell$ is defined in \eqref{tetaelle}, $A_N,$ $B_N$, $C_N,$ $D_N$  are positive constants and the higher order terms   
 $\Upsilon_\ell=\Upsilon_\ell (d_1, \ldots, d_\ell ) $ are  smooth functions  such that $ |\Upsilon_\ell|=o(1).$  

\end{itemize}
\end{proposition} 

\medskip

\begin{proof} [\bf {Proof of Theorem \ref{torri}}]
By (i) of Proposition \eqref{prob-rido},  it is sufficient to find a critical point of the reduced energy $\widetilde{J}_\eps$.   By (ii) of Proposition \eqref{prob-rido}, it is sufficient to find a critical point of the function
\begin{equation}\label{m-0}F_\eps(d_1,\dots,d_k):=\sum\limits_{\ell=1}^k\eps^{\theta_\ell}\(G_\ell(d_1,\dots,d_\ell)+o_\ell(1)\)\end{equation}
where
$$G_1(d_1):=-A_N |{\rm Weyl}_g(\xi)|^2_g d_1^4+B_N d_1^2$$
and $$ G_\ell(d_1,\dots,d_\ell):=-C_N \left(\frac{d_\ell}{d_{\ell-1}}\right)^{\frac{N-2}{2}} +B_N d_\ell^2\ \hbox{if}\ \ell=2,\dots,k.$$
Here $o_\ell(1)$ only depends on $d_1,\dots,d_\ell$ and $o_\ell(1)\to0$ as $\e\to0$  uniformly with respect to $(d_1, \dots, d_\ell) $ in compact subsets of $(0, +\infty)^\ell   .$
We shall prove that $F_\eps$ has a maximum point. The claim will follow.\\

First, the function $G_1$ has a unique critical point $d_1^*$ which is a global maximum. In particular, given $\delta>0$ there exists $\sigma_1>0$ such that
\begin{equation}\label{m-1}
G_1(d_1)\le G_1(d_1^*)-\delta\ \hbox{if}\ |d_1-d_1^*|=\sigma_1.
\end{equation}
Now,  for any $\ell=2,\dots,k$   the function  $d_\ell\to G_\ell(d_1^*,\dots,d_{\ell-1}^*,d_\ell)  $  has a unique
 critical point $d_\ell^*$ which is a global maximum. In particular, given $\delta>0$ there exists $\sigma_\ell>0$ such that
\begin{equation}\label{m-2}
G_\ell(d_1^*,\dots,d_{\ell-1}^*,d_\ell)\le G_\ell(d_1^*,\dots,d_{\ell-1}^*,d_\ell^*)-\delta\ \hbox{if}\ |d_\ell-d_\ell^*|=\sigma_\ell.
\end{equation}
 We consider the compact set $K:=[d_1^*-\sigma_1,d_1^*+\sigma_1]\times\dots\times[d_k^*-\sigma_k,d_k^*+\sigma_k].$
For any $\eps$ small enough, there exists a  $(d_1^\e,\dots,d_k^\e)\in K$ such that
 $F_\e(d_1^\e,\dots,d_k^\e):=\max _KF_\e.$ 
 First of all, let us prove that
 \begin{equation}\label{m-3}
 \lim\limits_{\e\to0} d_\ell^\eps=d_\ell^*\ \hbox{for any}\ \ell=1,\dots,k.
 \end{equation}
 Let us start with $\ell=1.$ We know that 
 $F_\e(d_1^\e,\dots,d_k^\e)\ge F_\e(d_1^*,d_2^\e,\dots,d_k^\e),$ then by \eqref{m-0} we deduce that
 $$\e^{\theta_1}\[G_1(d_1^\e)-G_1(d_1^*)+o(1)\]\ge 0,$$
 which implies
 $$G_1(d_1^\e)\ge G_1(d_1^*)+o(1).$$
 On the other hand, since $d_1^*$ is the maximum of $G_1$ we also have
 $$G_1(d_1^*)\ge G_1(d_1^\e).$$
Combining the two inequalities and passing to the limit we get $\lim\limits_{\e\to0}G_1(d_1^\e)=G_1(d_1^*)$ and so \eqref{m-3} follows.
Assume that \eqref{m-3} holds for $\ell=1,\dots,i-1$ and let us consider the case $ \ell=i.$
We know that
$$F_\e(d_1^\e,\dots,d_k^\e)\ge F_\e(d_1^*,\dots, d_i^*,d_{i+1}^\e\dots,d_k^\e),$$ then by \eqref{m-0} we deduce that
\begin{equation*}
\begin{aligned}&\e^{\theta_i}\[G_i(d_1^\e,\dots,d_i^\e)-G_i(d_1^*,\dots, d_i^*)+o(1)\]\\ &=\eps^{\theta_i}\[\underbrace{G_i(d_1^\e,\dots,d_i^\e)-G_i(d_1^*,\dots, d_{i-1}^*,d_i^\e)}_{=o(1)\ by\ \eqref{m-3}}+G_i(d_1^*,\dots, d_{i-1}^*,d_i^\e)-G_i(d_1^*,\dots, d_i^*)+o(1)\]\ge 0,\end{aligned}
\end{equation*}
 which implies
$$G_i(d_1^*,\dots, d_{i-1}^*,d_i^\e)-G_i(d_1^*,\dots, d_i^*)+o(1)\ge0.$$
On the other hand, since $d_i^*$ is the maximum of $G_i(d_1^*,\dots, d_{i-1}^*,\cdot)$ we also have
 $$ G_i(d_1^*,\dots, d_i^*)\ge G_i(d_1^*,\dots, d_{i-1}^*,d_i^\e).$$
Combining the two inequalities and passing to the limit we get $\lim\limits_{\e\to0}G_i(d_1^*,\dots, d_{i-1}^*,d_i^\e)=G_i(d_1^*,\dots, d_{i-1}^*,d_i ^*)$ and so \eqref{m-3} follows.\\

Now, let us prove that   if $\e$ is small enough  $(d_1^\e,\dots,d_k^\e)\not\in \partial K.$
Assume, that  $|d_i^\e-d_i^*|=\sigma_i.$ for some $i\ge1.$ 
We have
$$F_\e(d_1^\e,\dots, d_i^\e,\dots, d_k^\e)\ge F_\e(d_1^*,\dots, d_i^*,\dots,  d_k^\e).$$ 
 On the other hand, by \eqref{m-1} we deduce that
\begin{equation*}
\begin{aligned}&F_\e(d_1^\e,\dots,  d_i^\e,\dots, d_k^\e)- F_\e(d_1^*,\dots,d_i^*,\dots,   d_k^\e)\\
&=
 \e^{\theta_i}\[G_i(d_1^\e,\dots, d_i^\e)-G_i(d_1^*,\dots, d_i^*)+o(1)\]\\
 &= \e^{\theta_i}\[\underbrace{G_i(d_1^\e,\dots, d_i^\e)-G_i(d_1^*,\dots,d_{i-1}^*, d_i^\e)}_{=o(1)\ by\ \eqref{m-3}}+\underbrace{G_i(d_1^*,\dots,d_{i-1}^*, d_i^\e)-G_i(d_1^*,\dots, d_i^*)}_{<-\delta \ by\ \eqref{m-3}}+o(1)\]\\
 &<0\ \hbox{and a contradiction arises.}\end{aligned}
\end{equation*}

\end{proof}

\section{Appendix} \label{appA}
\begin{proof}[{\bf Proof of Lemma \ref{errorej}}]
When $\ell=1$ we argue exactly as in Lemma 3.1 of \cite{ep}.\\ 
Let us focus on the case $\ell\geq 2$. \\
It is useful to point out that by \eqref{V3} in geodesic coordinate
\begin{equation}\label{uvw}
|W_j(x)|\le c {\mu_j^{ \frac{N-2}{2}} \over\(\mu_j^2+|x|^2\)^{N-2\over2}},\ x\in B(0,r_0).
\end{equation}
 Since $\Pi^\bot_\ell[\i^*(\nu(\xi)Z_\ell^0)]=0$ then we have
\begin{equation}\label{normarj}
\begin{aligned}
\|\mathcal E_\ell\|&=\left\|\Pi^\bot_\ell\left\{ W_\ell-i^*\left[f\left(\sum_{i=1}^\ell  W_i\right)-f\left(\sum_{i=1}^{\ell-1} W_i\right)-\eps  W_\ell +\nu(\xi) Z_\ell^0\right]\right\}\right\|\\
&\leq c \left|-\mathcal L_g  W_\ell+ \eps   W_\ell-\nu(\xi) Z_\ell^0-f\left(\sum_{i=1}^\ell  W_i\right)+f\left(\sum_{i=1}^{\ell-1} W_i\right)\right|_{\frac{2N}{N+2}}\\
&\leq c \underbrace{\left|-\mathcal L_g  W_\ell+ \eps   W_\ell-\nu(\xi) Z_\ell^0-f( W_\ell )\right|_{\frac{2N}{N+2}}}_{(I)}\\
&+\underbrace{c\left|f\left(\sum_{i=1}^\ell  W_i\right)-f\left(\sum_{i=1}^{\ell-1} W_i\right)-f( W_\ell )\right|_{\frac{2N}{N+2}}}_{(II)}
\end{aligned}
\end{equation}
Arguing as in Lemma 3.1 of \cite{ep} we get that 
\begin{equation}\label{erroreRjparteI}
(I) \le \left\{\begin{aligned} &\mu_\ell^{\frac 52}+\eps \mu_\ell^2\qquad &\mbox{if}\,\, N=7\\ &\mu_\ell^{3}|\ln\mu_\ell|^{\frac 5 8}+\eps\mu_\ell^2\qquad &\mbox{if}\,\, N=8\\ &\mu_\ell^{3}+\eps\mu_\ell^2\qquad &\mbox{if}\,\, N\ge 9.\\ \end{aligned}\right.\end{equation}
Now, let us prove that
\begin{equation}\label{stimaii}
(II)=O\(\(\mu_\ell\over\mu_{\ell-1}\)^{N+2\over4}\)
\end{equation}
  For any $\ell=1, \ldots, k$ we introduce the set of disjoint annuli  
\begin{equation}\label{anelli}
\mathcal A_h := B_\xi(\sqrt{\mu_{h-1}\mu_h})\setminus B_\xi(\sqrt{\mu_h \mu_{h+1}}),\   h=1,\ldots, \ell 
\end{equation}
where we agree that $\mu_0:=\frac{r_0^2}{\mu_1}$ and $\mu_{\ell+1}:=0$. It is useful to point out that
 $  B_\xi(r_0)=\mathcal A_1\cup\dots\cup\mathcal A_\ell ,$ so all the bubbles $W_i$ are supported in $B_\xi(r_0).$
Therefore we have
$$(II)=\sum\limits_{h=1}^\ell\left|f\left(\sum_{i=1}^\ell  W_i\right)-f\left(\sum_{i=1}^{\ell-1} W_i\right)-f( W_\ell )\right|_{\frac{2N}{N+2},\mathcal A_h}.$$
 It is useful to remind that the choice of the $\mu_\ell$'s in \eqref{muj} implies that
$$\left( {\mu_\ell\over \mu_{\ell-1}}\right)^{\frac{N+2}{4}} =O\( \eps^{\frac p2 \theta_\ell}\).$$
 If $h=1, \ldots, \ell-1$ by Lemma \ref{yyl}  we have 
 \begin{equation*}
\begin{aligned}
|(II)|_ {\frac{2N}{N+2}, \mathcal A_h}&=\left|f\left(\sum_{i=1}^\ell  W_i\right)-f\left(\sum_{i=1}^{\ell-1} W_i\right)-f( W_\ell )\right|_ {\frac{2N}{N+2}, \mathcal A_h}\\ &\leq c\left|\left|\sum_{i=1}^{\ell-1}  W_i\right|^{p-1}| W_\ell |\right|_{\frac{2N}{N+2}, \mathcal A_h}+c\left| W_\ell^p\right|_{\frac{2N}{N+2}, \mathcal A_h}\\
&\leq c \sum_{i=1}^{\ell-1}\left| W_i^{p-1} W_\ell \right|_{\frac{2N}{N+2}, \mathcal A_h}+c\left| W_\ell\right|^p_{\frac{2N}{N-2}, \mathcal A_h}\\
&= O \(\left( {\mu_\ell\over \mu_{\ell-1}}\right)^{\frac{N+2}{4}}\),\end{aligned}\end{equation*}
because, by \eqref{uvw} we get for any $h=1, \ldots, \ell-1$
\begin{equation}\label{ok1}\begin{aligned}
\left| W_\ell\right| _{\frac{2N}{N-2}, \mathcal A_h}&\le   c\left[\int\limits_{ \sqrt{\mu_h\mu_{h+1}}\le |x|\le\sqrt{\mu_{h-1}\mu_{h}} }\frac{\mu_\ell^N}{(\mu_\ell^2+|x|^2)^N}\, dx\right]^{\frac{N-2}{2N}}\\
&=c\left[\int\limits_{ {\sqrt{\mu_h\mu_{h+1}}\over\mu_\ell}\le |x|\le{\sqrt{\mu_{h-1}\mu_h} \over\mu_\ell}}\frac{1}{(1+|y|^2)^N}\, dy\right]^{\frac{N-2}{2N}}\\ &=O \(\left(\frac{\mu_\ell}{\sqrt{\mu_h \mu_{h+1}}}\right)^{\frac{N-2}{2}}\)=O \(\left( {\mu_\ell\over \mu_{\ell-1}}\right)^{\frac{N-2}{4}}\)
\end{aligned}\end{equation}
and for any $h=1, \ldots, \ell-1$ and ${i=1,\dots,\ell-1}$
\begin{equation}\label{ok2}\begin{aligned}\left| |W_i|^{p-1} W_\ell \right|_{\frac{2N}{N+2}, \mathcal A_h}&\le c
\left\{\int\limits_{ \sqrt{\mu_h\mu_{h+1}}\le |x|\le\sqrt{\mu_{h-1}\mu_h} }\left[\frac{\mu_i^2}{(\mu_i^2+|x|^2)^2}\frac{\mu_\ell^{\frac{N-2}{2}}}{(\mu_\ell^2+|x|^2)^{\frac{N-2}{2}}}\right]^{\frac{2N}{N+2}}\, dx\right\}^{\frac{N+2}{2N}}\\
&\le c\left\{\int\limits_{ {\sqrt{\mu_h\mu_{h+1}}\over\mu_\ell}\le |x|\le{\sqrt{\mu_{h-1}\mu_h} \over\mu_\ell}}\mu_\ell^N\left[\frac{\mu_i^{-2} \mu_\ell^{-\frac{N-2}{2}} }{(1 +|y|^2)^{\frac{N-2}{2}}}\right]^{\frac{2N}{N+2}}\, dy\right\}^{\frac{N+2}{2N}}\\ &=O \left(\({\mu_\ell\over\mu_i}\)^2\( \frac{\mu_\ell}{\sqrt{\mu_h \mu_{h+1}}}\right)^{\frac{N-6}{2}}\)=O \(\left( {\mu_\ell\over \mu_{\ell-1}}\right)^{\frac{N+2}{4}}\).
\end{aligned}\end{equation}
 
If $h=\ell$ by Lemma \ref{yyl}  we have 
\begin{equation*}
\begin{aligned}
|(II)|_{\frac{2N}{N+2}, \mathcal A_\ell}&=\left|f\left(\sum_{i=1}^\ell  W_i\right)-f\left(\sum_{i=1}^{\ell-1} W_i\right)-f( W_\ell )\right|_{\frac{2N}{N+2}, \mathcal A_\ell}\\ &
\leq c\left| \sum_{i=1}^{\ell-1}   |W_\ell | ^{p-1}W_i\right|_{\frac{2N}{N+2}, \mathcal A_\ell}+c\left|\left|\sum_{i=1}^{\ell-1} W_i\right|^p\right|_{\frac{2N}{N+2}, \mathcal A_\ell}\\
&\le c\sum_{i=1}^{\ell-1} \left|  W_i | W_\ell |^{p-1}\right|_{\frac{2N}{N+2}, \mathcal A_\ell}+c\sum_{i=1}^{\ell-1} \left|W_i \right|^p_{\frac{2N}{N-2}, \mathcal A_\ell}\\
&=O \(\left( {\mu_\ell\over \mu_{\ell-1}}\right)^{\frac{N+2}{4}}\),
\end{aligned}\end{equation*}
because by \eqref{uvw} we get for any $i=1, \ldots, \ell-1$
\begin{equation}\label{ok3}
\begin{aligned}
 \left|W_i \right|_{\frac{2N}{N-2}, \mathcal A_\ell}&\le   c\left[\int\limits_{ \sqrt{\mu_\ell\mu_{\ell+1}}\le |x|\le\sqrt{\mu_{\ell-1}\mu_\ell} }\frac{\mu_i^N}{(\mu_i^2+|x|^2)^N}\, dx\right]^{\frac{N-2}{2N}}\\
&=c\left[\int\limits_{ {\sqrt{\mu_\ell\mu_{\ell+1}}\over\mu_i}\le |x|\le{\sqrt{\mu_{\ell-1}\mu_\ell} \over\mu_i}}\frac{1}{(1+|y|^2)^N}\, dy\right]^{\frac{N-2}{2N}}\\ &=O \(\left({\sqrt{\mu_{\ell-1} \mu_{\ell}}\over\mu_i}\right)^{\frac{N-2}{2}}\)=O \(\left( {\mu_\ell\over \mu_{\ell-1}}\right)^{\frac{N-2}{4}}\)
\end{aligned}\end{equation}
and  
\begin{equation}\label{ok4}\begin{aligned}\left| |W_\ell|^{p-1} W_i \right|_{\frac{2N}{N+2}, \mathcal A_\ell}&\le c
\left\{\int\limits_{ \sqrt{\mu_\ell\mu_{\ell+1}}\le |x|\le\sqrt{\mu_{\ell-1}\mu_\ell} }\left[\frac{\mu_\ell^2}{(\mu_\ell^2+|x|^2)^2}\frac{\mu_i^{\frac{N-2}{2}}}{(\mu_i^2+|x|^2)^{\frac{N-2}{2}}}\right]^{\frac{2N}{N+2}}\, dx\right\}^{\frac{N+2}{2N}}\\
&\le c\left\{\int\limits_{ {\sqrt{\mu_\ell\mu_{\ell+1}}\over\mu_i}\le |x|\le{\sqrt{\mu_{\ell-1}\mu_\ell} \over\mu_i}}\mu_i^N\left[ 
{\mu_\ell^{ 2} \mu_i^{-\frac{N-2}{2}}\over \mu_i^4  |y|^ 4} \right]^{\frac{2N}{N+2}}\, dy\right\}^{\frac{N+2}{2N}}\\ &=O \left(\({\mu_\ell\over\mu_i}\)^2\(  {\sqrt{\mu_{\ell-1} \mu_{\ell}}\over\mu_i} \right)^{\frac{N-6}{2}}\)=O \(\left( {\mu_\ell\over \mu_{\ell-1}}\right)^{\frac{N+2}{4}}\).
\end{aligned}\end{equation}

The claim follows collecting all the previous estimates.
\end{proof}
We recall the following useful lemma.
\begin{lemma}\label{yyl}
For any $a>0$ and $b\in\mathbb R$ we have
$$\left||a+b|^\beta- a^\beta\right|\le\left\{ 
\begin{aligned}&c(\beta)\min\{|b|^\beta,a^{\beta-1}|b|\}\ \hbox{if}\ 0<\beta<1\\
&c(\beta)\(|b|^\beta+a^{\beta-1}|b|\)\ \hbox{if}\ \beta>1\\
\end{aligned}
\right.$$
and
$$\left||a+b|^\beta(a+b)-a^{\beta+1}-(1+\beta)a^\beta b\right|\le\left\{ 
\begin{aligned}&c(\beta)\min\{|b|^{\beta+1},a^{\beta-1}b^2\}\ \hbox{if}\ 0<\beta<1\\
&c(\beta)\max\{|b|^{\beta+1},a^{\beta-1}b^2\}\ \hbox{if}\ \beta>1\\
\end{aligned}
\right.$$
\end{lemma}

\begin{proof}[{\bf Proof of Proposition \ref{phij}}] 
{\it Step 1: The case $\ell=1$} \\
(i) is trivial and (ii)   can be proved arguing exactly as in Proposition 3.1 of \cite{ep}.  \\
  Let us prove   that  (iii) holds.
The function $\phi_{1, \eps}$ weakly solves the first equation in \eqref{sistema}, namely $$\mathcal E_1+\mathcal S_1(\phi_{1, \eps})+\mathcal N_1(\phi_{1, \eps})=0.$$  
Then, if $\eps>0$ is small enough, there exists a constant $\lambda_\eps$ (depending only on $d_1$) such that $\phi_{1, \eps}$ weakly solves 
\begin{equation}\label{eqM}
-\mathcal L_g(W_1+\phi_{1, \eps})+ \eps(W_1+\phi_{1, \eps})-f(W_1+\phi_{1, \eps})=\lambda_ \eps \(-\mathcal L_g Z^0_1\)
\end{equation}
Let us first show that $\lambda_ \eps=o(1)$   as $\eps\to 0$.\\ We test the equation \eqref{eqM} by $Z^0_1$. We use the fact that $\phi_{1, \eps}\in \mathcal K_1^\bot$ and we get
\begin{equation}\label{eqMint1}
\begin{aligned}
&\int_M \[-\mathcal L_g  W_1 +\e W_1-f(W_1)\]Z^0_1\, d\nu_g  +\underbrace{\int_M (-\mathcal L_g \phi_{1,\eps} )Z_1^0\, d\nu_g}_{=\langle\phi_{1, \eps}, \iota ^* Z^0_1\rangle =0}+\e\int_M \phi_{1, \eps} Z^0_1\, d\nu_g\\
&-\int_M \[f(W_1+\phi_{1, \eps})-f(W_1)\]Z^0_1\, d\nu_g =\lambda_\eps \int_M \(-\mathcal L_g Z^0_1\)Z^0_1\, d\nu_g  
\end{aligned}
\end{equation}
Let us estimate each term in \eqref{eqMint1}.  
By   \eqref{stimaphi1}, \eqref{hhhh} and Lemma  \eqref{yyl} we deduce
$$\left|\int_M \[\mathcal L_g (W_1)+\e W_1-f(W_1)\]Z^0_1\, d\nu_g\right|\le c \left|\int_M \frac{\mu_1^{N-2\over 2}}{(\mu_1^2+ d_g(z, \xi)^2)^{N-2 \over 2}}Z^0_1\, d\nu_g\right|\le C \mu_1^2,$$
$$\e \left|\int_M \phi_{1, \e}Z^0_1\, d\nu_g\right|\le c \e\|\phi_{1, \e}\| |Z^0_1|_{2N\over N+2}\le c \e  \mu_1^2\|\phi_{1, \e}\|$$
and
\begin{equation*}
\begin{aligned}
 \left|\int_M \[f(W_1+\phi_{1, \eps})-f(W_1)\]Z^0_1\, d\nu_g\right|& \le \int_M | W_1^{p-1}\phi_{1, \e}Z^0_1|\,d\nu_g  
 + \int_M| \phi_{1, \e}^p Z^0_1|\, d\nu_g \le c\|\phi_{1, \e}\|.
\end{aligned}
\end{equation*}
Moreover, by \eqref{lb} and \eqref{gij} we deduce that for any $j=1,\dots,k$
\begin{equation}\label{ok100}
 \mathcal E_j^0:=-\mathcal L_g Z^0_j-f'(W_j)Z_j^0\quad\hbox{and}\quad | \mathcal E_j^0|=O\({\mu_j^{N-2\over2}\over\(\mu_j^2+d_g(z,\xi)^2\)^{N-2\over2}}\).
\end{equation}
Therefore an easy computation leads to 
\begin{equation*}
\begin{aligned}
 \int_M \(-\mathcal L_g Z^0_1\)Z^0_1\, d\nu_g & = \int_M f'\(W_1 \)(Z^0_1)^2\, d\nu_g  +\int_M \mathcal E_j^0 Z^0_1\, d\nu_g\\ &=\underbrace{\int\limits_{\mathbb R^N}f'(U)\(\psi^0\)^2dx}_{c_N}+O\(\mu_1^2\).
 \end{aligned}\end{equation*}

Collecting all the estimates in \eqref{eqMint1}, we deduce that $  \lambda_\eps =o(1).$ \\

  Let us set
 $\hat{u}_1 := W_1+\phi_{1, \eps} . $ 
 By \eqref{lb} and \eqref{ok100},  equation \eqref{eqM} in geodesic coordinates can be written as
\begin{equation}\label{eqM1}
-\Delta \hat{u}_1- (g^{ij}-\delta^{ij})\partial^2_{ij} \hat{u}_1- g^{ij}\Gamma^k_{ij} \partial_k \hat{u}_1+(\beta_N \R _g +\eps)\hat{u}_1-f(\hat{u}_1)= \lambda_\eps \[f'(W_1)Z^0_1+\mathcal E_1^0\]\ \hbox{in}\ B (0,r_0).
\end{equation}
 Therefore, if we take $r=\rho\mu_1$ and we scale  $\hat v_1(y):=\mu_1^{\frac{N-2}{2}}\hat u_1\circ {\rm exp}_\xi(\mu_1 y), $ the function $\hat v_1$ (taking into account \eqref{ok100}) solves  
\begin{equation}\label{eqM2}
\begin{aligned}
&  -\Delta \hat{v}_1-  \underbrace{\(g^{ij}(\mu_1 y)-\delta^{ij}(\mu_1 y)\)}_{:=a_{ij}(y)}\partial^2_{ij} \hat{v}_1-\underbrace{\mu_1 g^{ij}(\mu_1 y)\Gamma^k_{ij} (\mu_1 y)}_{:=b_k(y)} \partial_k \hat{v}_1\\
&+\underbrace{\mu_1^2(\beta_N \R _g(\mu_1 y) +\eps)}_{:=c(y)}\hat{v}_1-  \hat{v}_1^{p }=\underbrace{  \lambda_ \eps\( f'(U)\psi^0+O(\mu_1^2)\)}_{:=h(y)}\ \hbox{in}\  B\left(0,  {\rho} \right).
\end{aligned}
\end{equation}
By \eqref{gij} 
\begin{equation*}
\begin{aligned}&\sup\limits_{y\in B (0,  {\rho} )}|a_{ij}(y)|= O(\rho),\ \sup\limits_{y\in B (0,  {\rho} )}\(|\nabla a_{ij}(y)|+ |b_k(y)|\)= O(\mu_1),\ \sup\limits_{y\in B (0,  {\rho} )}|c(y)|= O(\mu_1^2),\\ &  \sup\limits_{y\in B (0,  {\rho} )} |h(y)|=O(\lambda_\e)=o(1).\end{aligned}
\end{equation*}
 We are in position to apply Proposition \ref{astratto}, which implies   that there exists $c>0$ such that
 $$ \sup\limits_{B (0,  {\rho} )}|\hat{v}_1|\le c.$$ 
  Therefore
   $$|\hat{v}_1(y)|=|U(y)+\mu_1^2 V(y) +\mu_1^{\frac{N-2}{2}}\hat\phi_{1, \eps}({\rm exp}_\xi (\mu_1 y))|\le c,\quad  y\in B\(0, \rho\)$$ and finally
$$ |\phi_{1, \e}(z)|\leq c \mu_1^{-\frac{N-2}{2}},\quad z\in B_\xi(\rho\mu_1).$$ 
{\it Step 2: The case $\ell\ge2.$}\\
Let us suppose that the first $(\ell-1)-$ th equations of \eqref{sistema} have the solutions $\phi_{j, \eps}$ with $j=1, \ldots, \ell-1$ with the all the  properties (i), (ii) and  (iii) and let us consider the $\ell$-th equation of \eqref{sistema}.
\begin{itemize}
\item {\it Proof of (i) and (ii):  existence and the uniform estimate}. \\

By  Proposition \ref{lineare} we can
 rewrite the equation $\mathcal E_\ell +\mathcal S_\ell(\phi_{\ell, \eps})+\mathcal N_\ell (\phi_{1, \eps}, \ldots, \phi_{\ell, \eps})=0$ as 
 $$\phi_{\ell, \eps} =\mathcal S^{-1}_\ell\left(\mathcal E_\ell +\mathcal N_\ell(\phi_{1, \eps}, \ldots, \phi_{\ell, \eps})\right):=\mathcal T_\ell(\phi_{\ell, \eps}).$$
As usual, we shall show that if $\eps$ is small enough, $\mathcal T_\ell:\mathcal B_\ell\to\mathcal B_\ell$ is a contraction mapping, where 
$\mathcal B_\ell:=\left\{\phi \in H^2_1(M)\,\, :\,\, \|\phi\|\leq R \eps^{\frac p2 \theta_\ell }\right\}$  for some   $R>0$.\\
First, by Proposition \ref{lineare} we   get
  $$\|\mathcal T_\ell(\phi_{\ell, \eps})\|\leq \|\mathcal S_\ell^{-1}\|\left(\|\mathcal E_\ell\|+ \|\mathcal N_\ell(\phi_{1, \eps}, \ldots, \phi_{\ell, \eps})\|\right)\leq c \left(\|\mathcal E_\ell\|+ \|\mathcal N_\ell(\phi_{1, \eps}, \ldots, \phi_{\ell, \eps})\|\right)$$ and by Proposition   \ref{errorej} we get
\begin{align*}
\|\mathcal E_\ell\| \leq c \eps^{\frac p2 {\theta_\ell} }.
\end{align*}
We shall prove that in the ball $\mathcal B_\ell$ there hold true that
\begin{equation}\label{cont1}\|\mathcal N_\ell(\phi_{1, \eps}, \ldots, \phi_{\ell, \eps})\|\le c \|\phi_{\ell, \eps}\|^p+c \|\phi_{1, \eps}\|^{p-1} \|\phi_{\ell, \eps}\|+c\eps^{\frac p2\theta_\ell}\end{equation}
and
\begin{equation}\label{cont2}\|\mathcal N_\ell(\phi_{1, \eps}, \ldots, \phi_{\ell, \eps})-\mathcal N_\ell(\phi_{1, \eps}, \ldots, \bar\phi_{\ell, \eps})\|\le   L\|\phi_{\ell, \eps}-\bar\phi_{\ell, \eps}\| \ \hbox{for some $L\in(0,1).$}\end{equation}
Then the claim will follow.\\

 We introduce the set of annuli defined in \eqref{anelli} and 
we get
\begin{equation}\label{Nl}
\begin{aligned}
\|\mathcal N_\ell(\phi_{1, \eps}, \ldots, \phi_{\ell, \eps})\| & \leq C \left|f\left(\sum_{j=1}^\ell ( W_j+\phi_{j, \eps})\right)-f\left(\sum_{j=1}^\ell   W_j\right)-f'\left(\sum_{j=1}^\ell  W_j\right)\phi_{\ell, \eps}\right.\\
&\left.-f\left(\sum_{j=1}^{\ell-1}( W_j+\phi_{j, \eps})\right)+f\left(\sum_{j=1}^{\ell-1} W_j\right)\right|_{\frac{2N}{N+2}}\\
&=\sum_{h=1}^{\ell }\left[\int_{\mathcal A_h} \left| \ldots\ldots\right|^{\frac{2N}{N+2}}\, d\nu_g \right]^{\frac{N+2}{2N}}.
\end{aligned}
\end{equation}
If $h=1,\dots,\ell-1$ by Lemma \ref{yyl} we deduce 
\begin{equation*}
\begin{aligned}
&\left[\int_{\mathcal A_h} \left| \ldots\ldots\right|^{\frac{2N}{N+2}}\, d\nu_g \right]^{\frac{N+2}{2N}}\le 
\left[\int_{\mathcal A_h} \left| f\left(\sum_{j=1}^\ell ( W_j+\phi_{j, \eps})\right)-f\left(\sum_{j=1}^\ell  W_j+\sum_{j=1}^{\ell-1}\phi_{j, \eps}\right)\right.\right.\\
&\left.\left.-f'\left(\sum_{j=1}^\ell  W_j+\sum_{j=1}^{\ell-1}\phi_{j, \eps}\right)\phi_{\ell, \eps}\right|^{\frac{2N}{N+2}}\, d\nu_g \right]^{\frac{N+2}{2N}}\\
&+\left[\int_{\mathcal A_h} \left| \left[f'\left(\sum_{j=1}^\ell  W_j+\sum_{j=1}^{\ell-1}\phi_{j, \eps}\right)-f'\left(\sum_{j=1}^\ell  W_j\right)\right]\phi_{\ell, \eps}\right|^{\frac{2N}{N+2}}\, d\nu_g \right]^{\frac{N+2}{2N}}\\
&+\left[\int_{\mathcal A_h} \left| f\left(\sum_{j=1}^\ell  W_j+\sum_{j=1}^{\ell-1}\phi_{j, \eps}\right)-f\left(\sum_{j=1}^{\ell-1}  W_j+\sum_{j=1}^{\ell-1}\phi_{j, \eps}\right)\right|^{\frac{2N}{N+2}}\, d\nu_g \right]^{\frac{N+2}{2N}}\\
&+\left[\int_{\mathcal A_h} \left| f\left(\sum_{j=1}^\ell  W_j\right)-f\left(\sum_{j=1}^{\ell-1}  W_j\right)\right|^{\frac{2N}{N+2}}\, d\nu_g \right]^{\frac{N+2}{2N}}\\
&\le c \|\phi_{\ell, \eps}\|^p+c \sum_{j=1}^{\ell-1}\|\phi_{j, \eps}\|^{p-1}\|\phi_{\ell, \eps}\|+c \left[\int_{\mathcal A_h}|  W_\ell|^{p+1}\, d \nu_g \right]^{\frac{N+2}{2N}}\\
&+c \left[\int_{\mathcal A_h}\left|\left|\sum_{j=1}^{\ell-1}( W_j+\phi_{j, \eps})\right|^{p-1}  W_\ell \right|^{\frac{2N}{N+2}}\, d\nu_g\right]^{\frac{N+2}{2N}}\\
&\le c \|\phi_{\ell, \eps}\|^p+c\sum_{j=1}^{\ell-1}\|\phi_{j, \eps}\|^{p-1}\|\phi_{\ell, \eps}\|+c|W_\ell|^p_{\frac{2N}{N-2}, \mathcal A_h}+c\sum_{j=1}^{\ell-1}||W_j|^{p-1}W_\ell| _{\frac{2N}{N+2}, \mathcal A_h}+\sum_{j=1}^{\ell-1}||\phi_j|^{p-1} W_\ell| _{\frac{2N}{N+2}, \mathcal A_h}\\
 &\le c \(\|\phi_{\ell, \eps}\|^p+  \|\phi_{1, \eps}\|^{p-1} \|\phi_{\ell, \eps}\| +\eps^{\frac p2\theta_\ell}\),
\end{aligned}
\end{equation*}
because of \eqref{ok1}, \eqref{ok2} and the following new estimate
\begin{equation}\label{ok5}
  ||\phi_j|^{p-1} W_\ell| _{\frac{2N}{N+2}, \mathcal A_h}=O\(\({\mu_\ell\over\mu_{\ell-1}}\)^{N+2\over4}\),\quad  j,h=1,\dots,\ell-1. \end{equation}
  Let us prove \eqref{ok5}.
We have to distinguish three cases $j\ge h+1$,  $j\le h-1$ and $j=h.$

If $j\le h-1 $  by \eqref{stimalinfty} we deduce that $\phi_j=O\(\mu_j^{-{N-2\over2}}\)$ in $\mathcal A_h\subset B(\xi,r\mu_j)$ and we have
\begin{equation}
\begin{aligned}
&||\phi_j|^{p-1} W_\ell| _{\frac{2N}{N+2}, \mathcal A_h}\le c{1\over \mu_j^2}|  W_\ell| _{\frac{2N}{N+2}, \mathcal A_h}
\le c{1\over \mu_j^2}\(meas \ \mathcal A_h\)^{2\over N}  |W_\ell| _{\frac{2N}{N-2}, \mathcal A_h}\\
& \le c{\mu_h\mu_{h+1}\over \mu_j^2} \left( {\mu_\ell\over \mu_{\ell-1}}\right)^{\frac{N-2}{4}} =O\( \left( {\mu_\ell\over \mu_{\ell-1}}\right)^{\frac{N+2}{4}} \).
	\end{aligned}
\end{equation}

If $h\le j\le \ell-1$ and $h+1\le \ell-1$ then by \eqref{ok1} and \eqref{stimaphi1} we deduce
\begin{equation}
\begin{aligned}
&||\phi_j|^{p-1} W_\ell| _{\frac{2N}{N+2}, \mathcal A_h}\le c \|\phi_j\|^{p-1}   |  W_\ell| _{\frac{2N}{N-2}, \mathcal A_h}
\le   c \|\phi_ {1}\|^{p-1}\left(\frac{\mu_\ell}{\sqrt{\mu_h \mu_{h+1}}}\right)^{\frac{N-2}{2}} \\
&\le c \|\phi_ {1}\|^{p-1}\left(\frac{\mu_\ell}{ \mu_{\ell-1}}\right)^{\frac{N-2}{2}} =O\( \left( {\mu_\ell\over \mu_{\ell-1}}\right)^{\frac{N+2}{4}} \)
	\end{aligned}
\end{equation}
  
If $j=h=\ell-1$ we split the annulus 
$$\mathcal A_{\ell-1}=\underbrace{\{\sqrt{\mu_{\ell-1}\mu_{\ell}}\le d_g(x,\xi)\le \mu_{\ell-1} \}}_{\mathcal A'_{\ell-1}}\cup
\underbrace{\{\mu_{\ell-1}  \le d_g(x,\xi)\le  \sqrt{\mu_{\ell-1}\mu_{\ell-2}}\}}_{\mathcal A''_{\ell-1}}$$
and we get
\begin{equation}
\begin{aligned}
&||\phi_{\ell-1}|^{p-1} W_\ell| _{\frac{2N}{N+2}, \mathcal A_{\ell-1}}= ||\phi_{\ell-1}|^{p-1} W_\ell| _{\frac{2N}{N+2}, \mathcal A'_{\ell-1}}
+ ||\phi_{\ell-1}|^{p-1} W_\ell| _{\frac{2N}{N+2}, \mathcal A''_{\ell-1}}\\
& \le c{1\over \mu_{\ell-1}^2}  |W_\ell| _{\frac{2N}{N+2}, \mathcal A'_{\ell-1}}
+c\|\phi_{\ell-1}\|^{p-1}   |  W_\ell| _{\frac{2N}{N-2}, \mathcal A''_{\ell-1}}\\ 
&=O \(\left( {\mu_\ell\over \mu_{\ell-1}}\right)^{\frac{N+2}{4}}\),
	\end{aligned}
\end{equation}
because
\begin{equation}\label{ok31}
\begin{aligned}
 \left|W_\ell\right|_{\frac{2N}{N+2}, \mathcal A'_{\ell-1}}&\le   c\left[\int\limits_{ \sqrt{\mu_{\ell-1}\mu_{\ell}}\le |x|\le \mu_{\ell-1} }\(\frac{\mu_\ell^{N-2\over2}}{(\mu_\ell^2+|x|^2)^{N-2\over2}}\)^{2N\over N+2}\, dx\right]^{\frac{N+2}{2N}}\\
&\le c\mu_\ell^{2}\left[\int\limits_{ {\sqrt{\mu_{\ell-1}\over\mu_{\ell}}}\le |y|\le{\mu_{\ell-1}\over\mu_\ell} }\(\frac{1}{(1+|y|^2)^{N-2\over2}}\)^{2N\over N+2}\, dy\right]^{\frac{N+2}{2N}}\\ &=O \(\mu_\ell^{2} \({\mu_\ell\over\mu_{\ell-1}}\)^{\frac{N-6}{4}}\)
\end{aligned}\end{equation}
and
\begin{equation}\label{ok32}
\begin{aligned}
 \left|W_\ell\right|_{\frac{2N}{N-2}, \mathcal A''_{\ell-1}}&\le   c\left[\int\limits_{\mu_{\ell-1}  \le |x|\le \sqrt{\mu_{\ell-1}\mu_{\ell-2}}}\(\frac{\mu_\ell^{N-2\over2}}{(\mu_\ell^2+|x|^2)^{N-2\over2}}\)^{2N\over N-2}\, dx\right]^{\frac{N-2}{2N}}\\
&\le c\left[\int\limits_{{\mu_{\ell-1}\over\mu_\ell}  \le |y|\le{ \sqrt{\mu_{\ell-1}\mu_{\ell-2}}\over\mu_\ell}}\frac{1}{(1+|y|^2)^{N}}\, dy\right]^{\frac{N-2}{2N}}\\ &=O \( \({\mu_\ell\over\mu_{\ell-1}}\)^{\frac{N-2}{2}}\)
\end{aligned}\end{equation}

It remains to evaluate the last term $h=\ell$ in \eqref{Nl}. By Lemma \eqref{yyl} we get
\begin{equation*}
\begin{aligned}
&\left[\int_{\mathcal A_\ell} \left| \ldots\ldots\right|^{\frac{2N}{N+2}}\, d\nu_g \right]^{\frac{N+2}{2N}}\le 
\left[\int_{\mathcal A_\ell} \left| f\left(\sum_{j=1}^\ell ( W_j+\phi_{j, \eps})\right)-f\left(\sum_{j=1}^\ell  W_j+\sum_{j=1}^{\ell-1}\phi_{j, \eps}\right)\right.\right.\\
&\left.\left.-f'\left(\sum_{j=1}^\ell  W_j+\sum_{j=1}^{\ell-1}\phi_{j, \eps}\right)\phi_{\ell, \eps}\right|^{\frac{2N}{N+2}}\, d\nu_g \right]^{\frac{N+2}{2N}}\\
&+\left[\int_{\mathcal A_\ell} \left| f\left(\sum_{j=1}^\ell  W_j+\sum_{j=1}^{\ell-1}\phi_{j, \eps}\right)-f\left(\sum_{j=1}^\ell  W_j\right)-f'\left(\sum_{j=1}^\ell  W_j\right)\sum_{j=1}^{\ell-1}\phi_{j, \eps}\right|^{\frac{2N}{N+2}}\, d\nu_g \right]^{\frac{N+2}{2N}}\\
&+\left[\int_{\mathcal A_\ell} \left| \left[f'\left(\sum_{j=1}^\ell  W_j+\sum_{j=1}^{\ell-1}\phi_{j, \eps}\right)-f'\left(\sum_{j=1}^\ell  W_j\right)\right]\phi_{\ell, \eps}\right|^{\frac{2N}{N+2}}\, d\nu_g \right]^{\frac{N+2}{2N}}\\
&+\left[\int_{\mathcal A_\ell} \left| f\left(\sum_{j=1}^{\ell-1}(  W_j+\phi_{j, \eps})\right)-f\left(\sum_{j=1}^{\ell-1}  W_j\right)-f'\left(\sum_{j=1}^{\ell-1}  W_j\right)\sum_{j=1}^{\ell-1} \phi_{j, \eps}\right|^{\frac{2N}{N+2}}\, d\nu_g \right]^{\frac{N+2}{2N}}\\
&+\left[\int_{\mathcal A_\ell} \left| \left[f'\left(\sum_{j=1}^\ell  W_j\right)-f'\left(\sum_{j=1}^{\ell-1}  W_j\right)\right]\sum_{j=1}^{\ell-1}\phi_{j, \eps}\right|^{\frac{2N}{N+2}}\, d\nu_g \right]^{\frac{N+2}{2N}}\\
&\le c \|\phi_{\ell, \eps}\|^p+ c \|\phi_{1, \eps}\|^{p-1}\|\phi_{\ell, \eps}\|+ c \sum_{j=1}^{\ell-1}\(\int_{\mathcal A_\ell}\left|\phi_{j, \eps}\right|^{p+1}\, d\nu_g\)^{\frac{N+2}{2N}}  +c\sum_{j=1}^{\ell-1}\(\int_{\mathcal A_\ell}\left|| W_\ell|^{p-1}\phi_{j, \eps}\right|^{\frac{2N}{N+2}}\, d\nu_g\)^{\frac{N+2}{2N}} \\
 &\le c \(\|\phi_{\ell, \eps}\|^p+  \|\phi_{1, \eps}\|^{p-1} \|\phi_{\ell, \eps}\| +\eps^{\frac p2\theta_\ell}\),
\end{aligned}
\end{equation*}
because if $j\le \ell-1 $  by \eqref{stimalinfty} we deduce that $\phi_j=O\(\mu_j^{-{N-2\over2}}\)$ in $\mathcal A_\ell\subset B(\xi,r\mu_j)$ and we have
\begin{equation*}
\(\int_{\mathcal A_\ell}\left|\phi_{j, \eps}\right|^{p+1}\, d\nu_g \)^{\frac{N+2}{2N}}\le c { \mu_j}^{-{N+2\over2}}\(\textrm{meas}\ \mathcal A_\ell\)^{\frac{N+2}{2N}}\le c \({\sqrt{\mu_\ell\mu_{\ell-1}}\over\mu_j}\)^{\frac{N+2}{2}}  =O\(\({\mu_\ell\over\mu_{\ell-1}}\)^{N+2\over4}\)
\end{equation*}
and
\begin{equation*}\begin{aligned}
 \(\int_{\mathcal A_\ell}\left|| W_\ell|^{p-1} \phi_{j, \eps}\right|^{\frac{2N}{N+2}}d\nu_g\)^{\frac{N+2}{2N}}&\le c  \mu_j^{-\frac{N-2}{2}}  
\(\int\limits_{\sqrt{\mu_\ell\mu_{\ell+1}}\le|x|\le\sqrt{\mu_{\ell-1}\mu_\ell}}\( {\mu_\ell^{ 2}\over\(\mu_\ell^2+|x|^2\)^{ 2}}\)^{2N\over N+2}dx\)^{\frac{N+2}{2N}}\\ & \le c  \mu_j^{-\frac{N-2}{2}}  \mu_\ell^2
\(\int\limits_{\sqrt{\mu_\ell\mu_{\ell+1}}\le|x|\le\sqrt{\mu_{\ell-1}\mu_\ell}}{1\over |x| ^ {8N\over N+2}}dx\)^{\frac{N+2}{2N}}\\ &
\le c\mu_j^{-\frac{N-2}{2}}  \mu_\ell^2\(\mu_{\ell-1}\mu_\ell\)^{N-6\over4}=O\(\({\mu_\ell\over\mu_{\ell-1}}\)^{N+2\over4}\).
\end{aligned}\end{equation*}
That concludes the proof of \eqref{cont1}.
Now, let us prove \eqref{cont2}.
 Again, by Lemma \ref{yyl} we get
\begin{equation*}
\begin{aligned}
&\|\mathcal N_\ell(\phi_{1, \eps}, \ldots, \phi_{\ell, \eps})-\mathcal N_\ell(\phi_{1, \eps}, \ldots, \bar\phi_{\ell, \eps})\| \\ & \leq c \left|f\left(\sum_{j=1}^{\ell }   W_j+\sum_{j=1}^{\ell-1}\phi_{j, \eps}+\phi_{\ell, \eps}\right)-f\left(\sum_{j=1}^{\ell }   W_j+\sum_{j=1}^{\ell-1}\phi_{j, \eps}+\bar\phi_{\ell, \eps}\right)-f'\left(\sum_{j=1}^\ell  W_j\right)\(\phi_{\ell, \eps}-\bar\phi_{\ell, \eps}\)
 \right|_{\frac{2N}{N+2}}\\
 \\ &= \left|\[f'\left(\sum_{j=1}^{\ell }   W_j+ \sum_{j=1}^{\ell-1}\phi_{j, \eps}+t\phi_{\ell, \eps}+(1-t)\bar\phi_{\ell, \eps}\right) -f'\left(\sum_{j=1}^\ell  W_j\right)\]\(\phi_{\ell, \eps}-\bar\phi_{\ell, \eps}\)\right|_{\frac{2N}{N+2}}\\
 &\le  c\left|\left|  \sum_{j=1}^{\ell-1}\phi_{j, \eps}+t\phi_{\ell, \eps}+(1-t)\bar\phi_{\ell, \eps}\right|^{p-1}\(\phi_{\ell, \eps}-\bar\phi_{\ell, \eps}\) \right|_{\frac{2N}{N+2}}\\
 &\le  c \sum_{j=1}^{\ell-1}\left|\left| \phi_{j, \eps}\right|^{p-1}\(\phi_{\ell, \eps}-\bar\phi_{\ell, \eps}\) \right|_{\frac{2N}{N+2}} 
 +c\left|\( |\phi_{\ell, \eps}|^{p-1}+ | \bar\phi_{\ell, \eps} |^{p-1}\)\(\phi_{\ell, \eps}-\bar\phi_{\ell, \eps}\) \right|_{\frac{2N}{N+2}} \\
&\le  c\(\sum_{j=1}^{\ell-1}\| \phi_{j, \eps}\|^{p-1}+\|\phi_{\ell, \eps}\|^{p-1}+\|\bar\phi_{\ell, \eps}\|^{p-1}\)\|\phi_{\ell, \eps}-\bar\phi_{\ell, \eps}\| \le L\|\phi_{\ell, \eps}-\bar\phi_{\ell, \eps}\| ,
\end{aligned}
\end{equation*}
for some $L\in(0,1)$ provided $\eps$ is small enough.\\

That concludes the proof.\\\\

\item {\it Proof of  (ii): the $C^1-$estimate.}\\

  We apply the Implicit Function Theorem to the map
 $\mathcal F_\e:(0,+\infty)^\ell\times H^1_g(M)\to \ H^1_g(M)$ defined by
\begin{equation*}
\begin{aligned}
\mathcal F_\e(\bar d,u):=&\phi+\pi({\bar d})\left\{\sum\limits_{i=1}^\ell W(d_i)-i^*\[f_\e\(\sum\limits_{i=1}^\ell W(d_i)+\pi({\bar d})(u)\)\]\right\}\\
&=\phi+\sum\limits_{i=1}^\ell W(d_i)- i^*\[f_\e\(\sum\limits_{i=1}^\ell W(d_i)+\pi({\bar d})(u)\) \]\\
&-\sum\limits_{i,j=1}^\ell\left< W(d_i ),Z(d_j) \right> Z(d_j)\\ &+\sum\limits_{j=1}^\ell \left<  i^*\[f_\e\(\sum\limits_{i=1}^\ell W(d_i)+\pi({\bar d})(u)\) \] ,Z(d_j) \right> Z(d_j)\\
\end{aligned}
\end{equation*}
where ${\bar d}:=(d_1,\dots,d_\ell)\in(0,+\infty)^\ell,$ 
$\pi({\bar d})(u):= u-\sum\limits_{i=1}^\ell\left<u,Z(d_i) \right> Z(d_i) ,$
  $Z(d_i) :=Z_i^0$ are defined in \eqref{Zj},  
$W(d_i):=W_i$ are defined in \eqref{Wj} and $f_\e(u):=f(u)-\e u.$

It is clear that $F_\e$ is a $C^1-$map. Moreover, 
by previous steps we deduce that for any $d_0\in(0,+\infty)^\ell$ there exists   $\phi_0=\sum\limits_{i=1}^\ell{\phi_{i,\e}} \in H^1_g$ such that
  $\pi({\bar d_0})(\phi_0)=\phi_0$ and
(see \eqref{eqMell})
 $\mathcal F_\e(\bar d_0,\phi_0)=0.$ We shall prove that
 \begin{equation}\label{dini1}
 \sup\limits_{u\not=0}{\|D\mathcal F_\e(\bar d_0,\phi_0)[u]\|\over \|u\|}\ge c>0 \end{equation}
and  \begin{equation}\label{dini2} \sup\limits_{d\not=0}{\|D\mathcal F_\e(\bar d_0,\phi_0)[\bar d]\|\over \|\bar d\|}=o(1)\ \hbox{as}\ \e\to0, \end{equation}
 {uniformly with respect to $\bar d_0$ in compact sets of $(0,\infty)^\ell$}.
The Implicit Function Theorem will imply that the map $\bar d_0\to \phi_0$ is a $C^1-$map and also that $|\nabla_{\bar d_0}\phi_0| =o(1).$
\\
First, we have
$$D\mathcal F_\e(\bar d_0,\phi_0)[u]=u-\pi({\bar d_0})\left\{ i^*\[f'_\e\(\sum\limits_{i=1}^\ell W(d_i^0)+\phi_0\) \pi({\bar d_0})(u) \]\right\} $$
and so
\begin{equation*}
\begin{aligned}
&\|D\mathcal F_\e(\bar d_0,\phi_0)[u]\| \ge c \left\|u-\pi({\bar d_0})(u)\right\| +	c\left\|\pi({\bar d_0})(u)-\pi({\bar d_0})\left\{ i^*\[f'_\e\(\sum\limits_{i=1}^\ell W(d_i^0)+\phi_0\) \pi({\bar d_0})(u) \]\right\}\right\| \\
&\ge c \left\|u-\pi({\bar d_0})(u)\right\|+c \underbrace{\left\|\pi({\bar d_0})(u)-\pi({\bar d_0})\left\{ i^*\[f' \(\sum\limits_{i=1}^\ell W(d_i^0) \) \pi({\bar d_0})(u) -\e\pi({\bar d_0})(u)  \]\right\}\right\|}_{\mathcal L
\(\pi({\bar d_0})(u)\) \ \hbox{(see \eqref{Lj})}} \\
&-\left\| \pi({\bar d_0})\circ i^*\left\{ \[f' \(\sum\limits_{i=1}^\ell W(d_i^0)+\phi_0\)-f' \(\sum\limits_{i=1}^\ell W(d_i^0) \) \]\pi({\bar d_0})(u) \right\}\right\|
 \\
&\ge  c \left\|u-\pi({\bar d_0})(u)\right\|+c\|\pi({\bar d_0})(u)\|-O\(\left|f' \(\sum\limits_{i=1}^\ell W(d_i^0)+\phi_0\)-f' \(\sum\limits_{i=1}^\ell W(d_i^0)\) \right|_{2N\over N+2}\|u  \|\) \\
&   \ge c\|u\|
 \end{aligned}
\end{equation*}
 because of Lemma \ref{lineare} and Lemma \ref{yyl}. Then \eqref{dini1} follows.\\
 
 Now,
  we can compute
\begin{equation*}
\begin{aligned}
&D\mathcal F_\e(\bar d_0,\phi_0)[\bar d]=    \sum\limits_{i=1}^\ell W'(d_i^0)d_i\\ &- i^*\[f'_\e\(\sum\limits_{i=1}^\ell W(d_i)+\phi_0\)\(\sum\limits_{i=1}^\ell W'(d_i^0)d_i
-\sum\limits_{i=1}^\ell\left<\phi_0, Z'(d_i^0)\right> Z (d_i^0)d_i- \sum\limits_{i=1}^\ell\left<\phi_0,Z (d_i^0)\right> Z'(d_i^0)d_i\)
\]\\
&-\sum\limits_{i,j=1}^\ell\[\left< W'(d_i ^0),Z(d_j^0) \right> Z(d_j^0) d_i +\left< W(d_i ^0),Z'(d_j^0) \right> Z(d_j^0)d_j+\left< W(d_i ^0),Z(d_j^0) \right> Z'(d_j^0)d_j\] \\
&+\sum\limits_{j=1}^\ell \left<  i^*\[f'_\e\(\sum\limits_{i=1}^\ell W(d_i^0)+\phi_0\)\right.\right.\\ &\qquad\left.\left.
\(\sum\limits_{i=1}^\ell W'(d_i^0)d_i
-\sum\limits_{i=1}^\ell\left<\phi_0, Z'(d_i^0)\right> Z (d_i^0)d_i- \sum\limits_{i=1}^\ell\left<\phi_0,Z (d_i^0)\right> Z'(d_i^0)d_i\)\] ,Z (d_j^0) \right> Z(d_j^0) \\
&+\sum\limits_{j=1}^\ell \left\{\left<  i^*\[f_\e\(\sum\limits_{i=1}^\ell W(d_i^0)+\phi_0\)\] ,Z'(d_j^0) \right> Z(d_j^0)d_j+\left<  i^*\[f_\e\(\sum\limits_{i=1}^\ell W(d_i^0)+\phi_0\)\] ,Z (d_j^0) \right> Z'(d_j^0)d_j\right\}\\ 
&=    \sum\limits_{i=1}^\ell \left\{W'(d_i^0)- i^*\[f'_\e\(\sum\limits_{i=1}^\ell W(d_i)+\phi_0\)W'(d_i^0)\]\right\}d_i\\ &- i^*\[f'_\e\(\sum\limits_{i=1}^\ell W(d_i)+\phi_0\)\(
-\sum\limits_{i=1}^\ell\left<\phi_0, Z'(d_i^0)\right> Z (d_i^0)d_i- \sum\limits_{i=1}^\ell\left<\phi_0,Z (d_i^0)\right> Z'(d_i^0)d_i\)
\]\\
&-\sum\limits_{i,j=1}^\ell \left< W'(d_i ^0)- i^*\[f'_\e\(\sum\limits_{i=1}^\ell W(d_i^0)+\phi_0\) 
  W'(d_i^0)\],Z (d_j^0) \right> Z(d_j^0)d_i\\
  &
+\sum\limits_{j=1}^\ell \left<  i^*\[f'_\e\(\sum\limits_{i=1}^\ell W(d_i^0)+\phi_0\) 
\( 
-\sum\limits_{i=1}^\ell\left<\phi_0, Z'(d_i^0)\right> Z (d_i^0)d_i- \sum\limits_{i=1}^\ell\left<\phi_0,Z (d_i^0)\right> Z'(d_i^0)d_i\)\] ,Z (d_j^0) \right> Z(d_j^0) \\
&-\sum\limits_{j=1}^\ell \left<\sum\limits_{i=1}^\ell  W(d_i^0)-i^*\[f_\e\(\sum\limits_{i=1}^\ell W(d_i^0)+\phi_0\)\] ,Z'(d_j^0) \right> Z(d_j^0)d_j\\
&-\sum\limits_{j=1}^\ell \left<\sum\limits_{i=1}^\ell  W(d_i^0)-i^*\[f_\e\(\sum\limits_{i=1}^\ell W(d_i^0)+\phi_0\)\] ,Z(d_j^0) \right> Z'(d_j^0)d_j
 \end{aligned}
\end{equation*}
and so
\begin{equation*}
\begin{aligned}
&\|D\mathcal F_\e(\bar d_0,\phi_0)[\bar d]\|= O\(   \sum_i  \left\|W'(d_i^0)- i^*\[f'_\e\(\sum\limits_{i=1}^\ell W(d_i)+\phi_0\)W'(d_i^0)\]\right\| \sum_i|d_i|\)\\ &+O\(\left|\sum_i W(d_i)+\phi_0\right|_{2N\over N-2}\|\phi_0\|\sum_i\|Z'(d_i^0)\|\sum_i\|Z(d_i^0)\|\sum_i|d_i|\)
\\
&+O\(   \sum_i  \left\|W'(d_i^0)- i^*\[f'_\e\(\sum\limits_{i=1}^\ell W(d_i)+\phi_0\)W'(d_i^0)\]\right\| \sum_i\|Z(d_i^0)\|^2\sum_i|d_i|\) \\
  & +O\(\left|\sum_i W(d_i)+\phi_0\right|_{2N\over N-2}\|\phi_0\|\sum_i\|Z'(d_i^0)\|\sum_i\|Z(d_i^0)\|^3\sum_i|d_i|\)
\\ &+O\(\left\| \sum\limits_{i=1}^\ell  W(d_i^0)-i^*\[f_\e\(\sum\limits_{i=1}^\ell W(d_i^0)+\phi_0\)\]  \right\| \sum_i\|Z'(d_i^0)\|\sum_i\|Z(d_i^0)\|\sum_i|d_i|\)\\
&=o(1)|\bar d|,
 \end{aligned}
\end{equation*}
because a straightforward computation shows that $\|Z'(d_i^0)\|=O(1)$ and $ \|Z(d_i^0)\|=O(1)$. Moreover,  taking into account Lemma \ref{yyl}, the estimate of the error in Lemma \ref{errorej} and estimate in \eqref{ok100} we get
$$\left\| \sum\limits_{i=1}^\ell  W(d_i^0)-i^*\[f_\e\(\sum\limits_{i=1}^\ell W(d_i^0)+\phi_0\)\]  \right\| =o(1)$$
and for any index $i$
$$ \left\|W'(d_i^0)- i^*\[f'_\e\(\sum\limits_{i=1}^\ell W(d_i)+\phi_0\)W'(d_i^0)\]\right\| =o(1).$$
Then \eqref{dini2} follows.\\

That concludes the proof.\\\\

\item {\it Proof of (iii): the pointwise estimate}.\\

 Let us consider the $j-$ th equation in \eqref{sistema} with $j<\ell$. Then, for any $\eps>0$ sufficiently small, there exist constants $\lambda_{j, 0}^\eps$ for $j=1, \ldots, \ell$ depending on $d_j$  for $j< \ell$ such that 
$$\begin{aligned}
&-\mathcal L_g( W_j+\phi_{j, \eps})+ \eps ( W_j+\phi_{j, \eps})-f\left(\sum_{m=1}^j( W_{m}+\phi_{m, \eps})\right)\\
&+f\(\sum_{m=1}^{j-1}(W_{m}+\phi_{m, \eps})\)=\sum_{m=1}^\ell \lambda_{m,  \eps} (-\mathcal L_g Z^0_m).
\end{aligned}
$$
If we sum on $j=1, \ldots, \ell$ we get

\begin{equation}\label{eqMell}
\begin{aligned}
&-\mathcal L_g\left(\sum_{j=1}^\ell(  W_j+\phi_{j, \eps})\right)+ \eps \left(\sum_{j=1}^\ell(  W_j+\phi_{j, \eps})\right)-f\left(\sum_{j=1}^\ell(  W_j+\phi_{j, \eps})\right) =\sum_{j=1}^\ell \lambda_{j, \eps} (-\mathcal L_g Z^0_j)
\end{aligned}
\end{equation}
First, we prove that $ \lambda_{j,\eps}=o(1)$   as $\eps\to 0$.\\
We test the equation \eqref{eqM} by $Z^0_i$ for $i=1,\dots,\ell$. We use the fact that each $\phi_{j, \eps}\in \mathcal K_\ell^\bot$ and we get
\begin{equation}\label{eqMint1-2}
\begin{aligned}
&\sum_{j=1}^\ell\int_M \[-\mathcal L_g  W_j +\e W_j-f(W_j)\]Z^0_i\, d\nu_g  +\sum_{j=1}^\ell\underbrace{\int_M (-\mathcal L_g \phi_{j, \eps} )Z_i^0\, d\nu_g}_{=\langle\phi_{j, \eps}, Z^0_1\rangle_g=0}+\e\sum_{j=1}^\ell\int_M \phi_{j, \eps} Z^0_i\, d\nu_g\\
&-\int_M \[f\(\sum_{j=1}^\ell W_j+\phi_{j, \eps}\)-\sum_{j=1}^\ell f(W_j)\]Z^0_i\, d\nu_g =\sum_{j=1}^\ell \lambda_{j, \eps}\int_M (-\mathcal L_g Z^0_j)Z_i^0\, d\nu_g  
\end{aligned}
\end{equation}
Let us estimate each term in \eqref{eqMint1}.  
By   \eqref{stimaphi1} and \eqref{hhhh}  we deduce
$$\left|\int_M \[\mathcal L_g (W_j)+\e W_j-f(W_j)\]Z^0_i\, d\nu_g\right|\le c \left|\int_M \frac{\mu_j^{N-2\over 2}}{(\mu_j^2+ d_g(z, \xi)^2)^{N-2 \over 2}}Z^0_i\, d\nu_g\right|\le C \mu_i^2,$$
$$\e \left|\int_M \phi_{j, \e}Z^0_i\, d\nu_g\right|\le c \e\|\phi_{j, \e}\| |Z^0_i|_{2N\over N+2}\le c \e  \mu_i^2\|\phi_{1, \e}\|$$
Moreover, we have
\begin{equation*}
\begin{aligned}
 \int_M \[f\(\sum_{j=1}^\ell W_j+\phi_{j, \eps}\)-\sum_{j=1}^\ell f(W_j)\]Z^0_i\, d\nu_g  &  = \int_M \[f\(\sum_{j=1}^\ell W_j+\phi_{j, \eps}\)- f\(\sum_{j=1}^\ell W_j\)\]Z^0_i\, d\nu_g  
 \\ &+\sum\limits_{\kappa=1}^\ell\int_M\[f\(\sum_{j=1}^\kappa W_j \)-f\(\sum_{j=1}^{\kappa-1} W_j\)-f(W_\kappa)\]Z^0_i\, d\nu_g  \\
 &=o(1),
\end{aligned}
\end{equation*}
because by Lemma  \eqref{yyl}
\begin{equation*}
\begin{aligned}
\left|\int_M \[f\(\sum_{j=1}^\ell W_j+\phi_{j, \eps}\)-\sum_{j=1}^\ell f(W_j)\]Z^0_i\, d\nu_g \right|& \le \sum_{j,m=1}^\ell  \int_M | W_j^{p-1}\phi_{m, \e}Z^0_i|\,d\nu_g  
 +\sum_{j =1}^\ell \int_M| \phi_{j, \e}^p Z^0_i|\, d\nu_g\\ & \le c\|\phi_{1, \e}\|=o(1)
\end{aligned}
\end{equation*}
and by  \eqref{stimaii}described
\begin{equation*}
\begin{aligned}\int_M\[f\(\sum_{j=1}^\kappa W_j \)-f\(\sum_{j=1}^{\kappa-1} W_j\)-f(W_\kappa)\]Z^0_i\, d\nu_g  &\le \left|f\(\sum_{j=1}^\kappa W_j \)-f\(\sum_{j=1}^{\kappa-1} W_j\)-f(W_\kappa)\right|_{2N\over N+2}|Z^0_i|_{2N\over N-2}\\ &=o(1).\end{aligned}
\end{equation*}
Moreover, by \eqref{ok100} we deduce that
\begin{equation}\label{czero}
\begin{aligned}
 \int_M \(-\mathcal L_g Z^0_j\)Z^0_i\, d\nu_g  = \int_M f'\(W_j \) Z^0_jZ^0_i\, d\nu_g  +\int_M \mathcal E_j^0 Z^0_i\, d\nu_g= c_0\(\delta_{ij}+o(1)\)
 \end{aligned}\end{equation}
 where $c_0$ is defined as follows.
 Indeed, by \eqref{ok100} and Holder inequality
 \begin{equation*}\begin{aligned}\int_M \mathcal E_j^0 Z^0_i\, d\nu_g&=O\(\int_{B(0,r_0)}{\mu_j^{N-2\over2}\over\(\mu_j^2+|x|^2\)^{N-2\over2}}{\mu_i^{N-2\over2}\over\(\mu_i^2+|x|^2\)^{N-2\over2}}dx\)\\ &=O\(\left|{\mu_j^{N-2\over2}\over\(\mu_j^2+|x|^2\)^{N-2\over2}}\right|_{2N\over N-2}\left|{\mu_i^{N-2\over2}\over\(\mu_i^2+|x|^2\)^{N-2\over2}}\right|_{2N\over N+2}\)=O(\mu_i^2) 
\end{aligned}\end{equation*}
and 
\begin{equation*}\begin{aligned}
\int_M f'\(W_j \) (Z^0_j)^2\, d\nu_g   =
\underbrace{\int_{\mathbb R^N}f'(U)(\psi^0)^2dx}_{=c_0}+o(1) 
\end{aligned}\end{equation*}
and if $j\not=i$
\begin{equation*}\begin{aligned}
\int_M f'\(W_j \) Z^0_jZ^0_i\, d\nu_g & =O\(\int_{B(0,r_0)}{\mu_j^{N+2\over2}\over\(\mu_j^2+|x|^2\)^{N+2\over2}}{\mu_i^{N-2\over2}\over\(\mu_i^2+|x|^2\)^{N-2\over2}}dx\)
\\ &=\left\{
\begin{aligned}&O\(\int_{B(0,r_0)}{\mu_j^{N+2\over2}\over\(\mu_j^2+|x|^2\)^{N+2\over2}} \mu_i^{-{N-2\over2 }}dx\)=O\(\({\mu_j\over\mu_i}\)^{N-2\over 2}\)\ &\hbox{if}\ j>i\\
&O\(\int_{B(0,r_0)}{\mu_j^{N+2\over2}\over\(\mu_j^2+|x|^2\)^{N+2\over2}}{\mu_i^{N-2\over2}\over |x| ^{N-2 }}dx\)=O\(\({\mu_i\over\mu_j}\)^{N-2\over 2}\)\ &\hbox{if}\ j<i .\\ \end{aligned}\right.
\end{aligned}\end{equation*}

Therefore an easy computation leads to 
\begin{equation*}
\begin{aligned}
 \int_M \(-\mathcal L_g Z^0_1\)Z^0_1\, d\nu_g  = \int_M f'\(W_1 \)(Z^0_1)^2\, d\nu_g  +\int_M \mathcal E_j^0 Z^0_1\, d\nu_g=\underbrace{\int\limits_{\mathbb R^N}f'(U)\(\psi^0\)^2dx}_{c_N}+O\(\mu_1^2\).
 \end{aligned}\end{equation*}
 Collecting all the previous estimates we get
that each $\lambda_{i,\e}=o(1)$. That proves our first claim.\\

Now, let us set  $\hat{u}_\ell := \sum_{j=1}^\ell(  W_j+\phi_{j, \eps}).$ By  \eqref{lb} and \eqref{ok100},    equation \eqref{eqMell} in geodesic coordinates can be written as
\begin{equation}\label{eqM1-2}
\begin{aligned}
&-\Delta \hat{u}_\ell-  (g^{ij}-\delta^{ij})\partial^2_{ij} \hat{u}_\ell-g^{ij}\Gamma^k_{ij} \partial_k \hat{u}_\ell+(\beta_N \R _g +\eps)\hat{u}_\ell-f\(\hat u_\ell\)\\ &= \sum_{j=1}^{\ell}\lambda_{  j,\eps}  \[f'(W_j)Z^0_j+\mathcal E_j^0\]\qquad \hbox{in}\ B (0,r_0).
\end{aligned}
\end{equation}
Therefore, if we take $r=\rho\mu_\ell$ and we scale  $\hat v_\ell(y):=\mu_\ell^{\frac{N-2}{2}}\hat \hat{u}_\ell\circ {\rm exp}_\xi(\mu_\ell y ),$ the function  $\hat v_\ell$ solves
\begin{equation}\label{eqM2ell}
\begin{aligned}
&-\Delta \hat{v}_\ell-  \underbrace{(g^{ij}(\mu_\ell y)-\delta^{ij}(\mu_\ell y))}_{:=a_{ij}(y)}\partial^2_{ij} \hat{v}_\ell-\underbrace{ \mu_\ell g^{ij}(\mu_\ell y)\Gamma^k_{ij} (\mu_\ell y) }_{:=b_k(y)}\partial_k \hat{v}_\ell\\
&+\underbrace{\mu_\ell^2(\beta_N \R _g(\mu_\ell y) +\eps) }_{:=c(y)} \hat v_\ell- \hat v_\ell^{p}\\ &=\underbrace{\lambda_{\ell,\e} \[f'(U)\psi^0 +O(\mu_\ell^2)  \]+\sum_{j=1}^{\ell-1}\lambda_{  j,\eps}   \hat{\mathcal E}_j^0}_{:=h(y)} \qquad \hbox{in}\ B (0,\rho),
\end{aligned}
\end{equation}
 where by \eqref{ok100} we easily deduce that 
$$\hat{\mathcal E}_j^0=O\(\mu_\ell^{N+2\over2}\mu_j^{\frac{N+2}{2}} \over\(\mu_j^2+|\mu_\ell y|^2\)^{N+2\over2}\)=O\(\({\mu_\ell \over\mu_j }\)^{\frac{N+2}{2}} \)=o(1).$$ \\
By \eqref{gij} 
\begin{equation*}
\begin{aligned}
 &\sup\limits_{y\in B (0,  {\rho} )}|a_{ij}(y)|= O(\rho),\ \sup\limits_{y\in B (0,  {\rho} )}\(|\nabla a_{ij}(y)|+ |b_k(y)|\)= O(\mu_\ell),\  \sup\limits_{y\in B (0,  {\rho} )}|c(y)| = O(\mu_\ell^2)\\ &\sup\limits_{y\in B (0,  {\rho} )} |h(y)|= o(1). \end{aligned}\end{equation*}
We are in position to apply Proposition \ref{astratto}, which implies   that there exists $c>0$ such that
 $$ \sup\limits_{B (0,  {\rho} )}|\hat{v}_\ell|\le c.$$
 Therefore
 $$|\hat{v}_\ell(y)|=\left| \sum_{j=1}^{\ell-1} \(\hat  W_j+\hat\phi_{j, \eps} \) +U(y)+\mu_\ell^2 V(y) +\mu_\ell^{\frac{N-2}{2}}\hat\phi_{\ell, \eps}({\rm exp}_\xi (\mu_1 y))\right|\leq c\quad \mbox{if}\ y\in B\(0, \rho\)$$ and this implies that 
$$|\phi_{\ell, \eps} (z)|\leq c \mu_\ell^{-\frac{N-2}{2}}\quad \mbox{if}\ z\in  B_\xi(\rho\mu_l).$$

\end{itemize}
\end{proof}
 
\begin{proposition}\label{astratto}
Let $u\in W^{1,2}(B(0,r))$ be a solution of
\begin{equation}\label{operatore}
-\Delta u +\sum_{i, j} a_{ij}\partial^2_{ij} u +\sum_\ell b_\ell\partial_\ell u + c u-u^p= h \qquad \mbox{in}\quad B(0, r).
\end{equation}
Assume that  there exist  $\lambda$ positive and small enough and $c>0$ such that 
\begin{equation}\label{ipotesi}
\max_{i, j}|a_{ij}|_{\infty, B(0, r)}\le \lambda;\qquad |\nabla a_{ij}|_{\infty, B(0, r )}+|b_\ell|_{\infty, B(0, r )}\le c;\qquad |h|_{\infty, B(0, r )}\le c.
\end{equation}
Then, if $\rho<r/2$
\begin{equation*}
|u|_{\infty, B(0,\rho )} \le C
\end{equation*} 
for some positive constant $C.$
\end{proposition}
\begin{proof}
The proof
 relies on a boot-strap argument as in Lemma 6 of
  \cite{h} together with standard elliptic estimates as in Theorem 8.17 of \cite{gt}.
\end{proof}

\begin{proof}[{\bf Proof of Proposition \ref{prob-rido}}]

{\it  Proof of (i).}\\
Let us prove that if $(d_1,\dots,d_k)$ is a critical point of $\widetilde J_\e$ then $ \sum_{\ell=1}^k ( W_\ell +\phi_{\ell, \e})$ is a critical point of the functional $J_\e.$
We have
$$
0=\partial_{d_h}\widetilde J_\e(d_1,\dots,d_k)=\nabla J_\e\( \sum_{\ell=1}^k ( W_\ell +\phi_{\ell, \e})\) \[\partial _{d_h} W_h+\partial_{d_h} \sum_{\ell=h}^k\phi_\ell\], \ \hbox{for any}\ h=1,\dots,k.
$$
Since $$\nabla J_\e\( \sum_{\ell=1}^k ( W_\ell +\phi_{\ell, \e})\) =\sum\limits_{j=1}^k \lambda_{j,\e}Z^0_j,$$
we get
$$0=\sum\limits_{j=1}^k \lambda_{j,\e}\<Z^0_j,\partial _{d_h} W_h+\partial_{d_h} \sum_{\ell=h}^k\phi_\ell\>, \ \hbox{for any}\ h=1,\dots,k.$$
Now,
$$\partial _{d_h} W_h=\e^{\gamma_h} {1\over \mu_h} \(Z^0_h+o(1)\)={1\over d_h} \(Z^0_h+o(1)\).$$
Moreover, by \eqref{blabla} we get
$$\<Z^0_j, \partial_{d_h}\phi_\ell\>=O\(\|Z^0_j\|\cdot\|\partial_{d_h}\phi_\ell\|\)=o\(1\). $$
Finally, by \eqref{czero} we get
$$\<Z^0_j,Z^0_h\>=o(1)\ \hbox{if}\ j\not=h\ \hbox{and}\ \<Z^0_h,Z^0_h\>=c_0+o(1).$$
Therefore,   the matrix relative to the system of the $\lambda_{j,\e}$'s is diagonally dominant and so   each $\lambda_{j,\e}$ is equal to zero. That proves our claim.\\\\

{\it  Proof of (ii).}\\

{\it Step 1.}
Let us first show that
$$J_\e\( \sum_{\ell=1}^k ( W_\ell +\phi_{\ell, \e})\) = J_\e \(\sum_{\ell=1}^k  W_\ell \) +  \sum_{\ell=1}^k \e^{\theta_\ell} \Upsilon_\ell $$ where $\Upsilon_\ell= \Upsilon_\ell (d_1, \ldots, d_\ell) $ are smooth functions such that $|\Upsilon_\ell |= o(1)$
  for any $\ell=1, \ldots, k$.
Indeed:
\begin{equation}\label{funz}
\begin{aligned}
&J_\e\(\sum_{\ell=1}^k ( W_\ell +\phi_{\ell, \eps})\)-J_\e\(\sum_{\ell=1}^k  W_\ell \)=\frac 12 \sum_{\ell=1}^k\int_M |\nabla_g \phi_{\ell, \eps}|^2\, d\nu_g+\sum_{m>\ell}\int_M \nabla_g \phi_{\ell, \eps}\nabla_g\phi_{m, \eps}\, d\nu_g\\
& +\frac 12 \sum_{\ell=1}^k\int_M \(\beta_N \R _g +\e\) \phi_{\ell, \eps}^2\, d\nu_g+\sum_{m>\ell}\int_M \(\beta_N \R _g +\e\)\phi_{\ell, \eps}\phi_{m, \eps}\, d\nu_g \\
&+\int_M \nabla_g \(\sum_{\ell=1}^k  W_\ell \)\nabla_g \(\sum_{m=1}^k\phi_{m, \eps}\)\, d\nu_g+\int_M \(\beta_N \R _g +\e\)\(\sum_{\ell=1}^k W_\ell \)\(\sum_{m=1}^k\phi_{m, \eps}\)\, d\nu_g\\
&-\int_M \[F\(\sum_{\ell=1}^k( W_\ell +\phi_{\ell, \eps})\)-F\(\sum_{\ell=1}^k W_\ell \)\]\, d\nu_g.
\end{aligned}
\end{equation}
Since each function $\phi_{\ell, \eps}$ solves the  equation
\begin{equation*}
-\mathcal L_g ( W_\ell +\phi_{\ell, \eps})+  \e ( W_\ell +\phi_{\ell, \eps})= f\(\sum_{j=1}^\ell ( W_j+\phi_{j, \e})\)-f\(\sum_{j=1}^{\ell-1} ( W_j+\phi_{j, \e})\)+\mathcal L_g\psi
\end{equation*}
 for some $\psi\in \mathcal K_{\mu_\ell, \xi},$
if we multiply    by $\phi_{m, \eps}$ with $m\ge \ell$ taking into account that   $\phi_{m, \eps} \bot \psi,$  we get
\begin{equation*}
\begin{aligned}
&\int_M \nabla_g  W_\ell  \nabla_g \phi_{m, \eps}\, d\nu_g + \int_M \nabla_g \phi_{\ell, \eps} \nabla_g \phi_{m, \eps}\, d\nu_g+\int_M \(\beta_N \R _g +\e\) W_\ell \phi_{m, \eps}\, d\nu_g \\
&+\int_M \(\beta_N \R _g +\e\)\phi_{\ell, \eps}\phi_{m, \eps}\, d\nu_g= \int_M \[ f\(\sum_{j=1}^\ell ( W_j+\phi_{j, \e})\)-f\(\sum_{j=1}^{\ell-1} ( W_j+\phi_{j, \e})\)\]\phi_{m, \eps}\, d\nu_g.
\end{aligned}
\end{equation*}
Therefore,  \eqref{funz} reads as
\begin{equation*}
\begin{aligned}
&\hskip-1,0 cm J_\e\(\sum_{\ell=1}^k ( W_\ell +\phi_{\ell, \eps})\)-J_\e\(\sum_{\ell=1}^k  W_\ell \)=-\frac 12 \sum_{\ell=1}^k\|\phi_{\ell, \eps}\|^2 -\frac 12\e \sum_{\ell=1}^k|\phi_{\ell, \eps}|_2^2\\
&+ \underbrace{\sum_{\ell=1}^k\int_M f\(\sum_{j=1}^\ell ( W_j+\phi_{j, \e})\) \phi_{\ell, \eps}\, d\nu_g -\sum_{\ell=1}^k\int_M f\(\sum_{j=1}^{\ell-1} ( W_j+\phi_{j, \e})\) \phi_{\ell, \eps}\, d\nu_g}_{(1)}\\
&+\underbrace{\sum_{m>\ell}\int_M f\(\sum_{j=1}^\ell ( W_j+\phi_{j, \e})\) \phi_{m, \eps}\, d\nu_g - \sum_{m>\ell} f\(\sum_{j=1}^{\ell-1} ( W_j+\phi_{j, \e})\) \phi_{m, \eps}\, d\nu_g}_{(2)}\\
& +\underbrace{\sum_{m<\ell}\int_M \[\mathcal L_g ( W_\ell )+\e  W_\ell  -f( W_\ell )\]\phi_{m, \eps}\, d\nu_g}_{(3)} +\underbrace{\sum_{m<\ell} \int_M f( W_\ell )\phi_{m, \eps}\, d\nu_g}_{(4)} \\
&\underbrace{-\int_M \[F\(\sum_{\ell=1}^k( W_\ell +\phi_{\ell, \eps})\)-F\(\sum_{\ell=1}^k W_\ell \)\]\, d\nu_g}_{(5)}
\end{aligned}
\end{equation*}
First, $(3)$ only depends on $d_1,\dots,d_\ell$ and
\begin{equation*}
|(3)| \leq \left[\int_M \(\frac{\mu_\ell^{N-2 \over 2}}{(\mu_\ell^2 +d_g(z, \xi)^2)^{\frac{N-2}{2}}}\)^{\frac{N+2}{2}}\, d\nu_g\right]^{\frac{2N}{N+2}}\|\phi_{1, \e}\|\le c \mu_\ell^2 \e^{1+\tau} =o(\e^{\theta_\ell}),
\end{equation*}
because by \eqref{lb} and \eqref{gij} we easily deduce that
   \begin{equation}\label{hhhh}|-\mathcal L_g ( W_\ell )+\e  W_\ell  -f( W_\ell )|\le c \frac{\mu_\ell^{{N-2 \over 2}}}{(\mu_\ell^2 +d_g(z, \xi)^2)^{\frac{N-2}{2}}}.\end{equation}
 Next, we remark that
\begin{equation*}
(2):=\sum_{\ell=2}^k \int_M f\(\sum_{j=1}^{\ell-1}( W_j+\phi_{j, \e})\)\phi_{\ell, \eps}\, d\nu_g
\ \hbox{and}\
(4):=\sum_{\ell=2}^k \int_M f\( W_\ell \)\sum_{j=1}^{\ell-1}\phi_{\ell, \eps}\, d\nu_g
\end{equation*}
and so
\begin{equation*}
\begin{aligned}
 &(1)+(2)+(4)+(5)\\ &=  -\int_M\[ F(W_1+\phi_{1, \e})-F(W_1)-f(W_1+\phi_{1, \e})\phi_{1, \e}\]\, d\nu_g \\
&-\sum_{\ell=2}^k \int_M \[ F\(\sum_{j=1}^{\ell}( W_j+\phi_{j, \e})\)-F\(\sum_{j=1}^{\ell-1}( W_j+\phi_{j, \e})\)\]\, d\nu_g \\
& +\sum_{\ell=2}^k \int_M \[ F\(\sum_{j=1}^{\ell} W_j\)-F\(\sum_{j=1}^{\ell-1} W_j\)\]\, d\nu_g+\sum_{\ell=2}^k \int_M f\(\sum_{j=1}^\ell ( W_j+\phi_{j, \e})\)\phi_{\ell, \eps}\, d\nu_g \\
& +\sum_{\ell=2}^k \int_M f( W_\ell )\sum_{j=1}^{\ell-1}\phi_{j, \e}\, d\nu_g \\
&= \int_M a_1\, d\nu_g   + \sum_{\ell=2}^k  \int_M a_\ell\, d\nu_g
\end{aligned}
\end{equation*}
where
$$   a_1 :=-\[ F(W_1+\phi_{1, \e})-F(W_1)-f(W_1+\phi_{1, \e})\phi_{1, \e}\]$$
and for any $\ell=2,\dots,k$
\begin{equation*}
\begin{aligned}
 a_\ell  &:=- \[ F\(\sum_{j=1}^{\ell}( W_j+\phi_{j, \e})\)-F\(\sum_{j=1}^{\ell-1}( W_j+\phi_{j, \e})\)\] \\
& + \[ F\(\sum_{j=1}^{\ell} W_j\)-F\(\sum_{j=1}^{\ell-1} W_j\)\] + f\(\sum_{j=1}^\ell ( W_j+\phi_{j, \e})\)\phi_{\ell, \eps} \\
& +  f( W_\ell )\sum_{j=1}^{\ell-1}\phi_{j, \e}.  \\
\end{aligned}
\end{equation*}
It is clear that each $a_\ell$'s only depends on $d_1,\dots,d_\ell$.
Moreover, by Lemma \ref{yyl} it follows that
\begin{equation*}
\begin{aligned}
 \int_M a_1\, d\nu_g&:= -\int_M \[F(W_1+\phi_{1, \e})-F(W_1)-F'(W_1)\phi_{1, \e}\]\, d\nu_g\\ &+\int_M \[F'(W_1+\phi_{1, \e})-F'(W_1)\phi_{1, \e}\]\, d\nu_g \\
&=O\(\|\phi_{1, \e}\|^2\)=o\(\eps^2\).
\end{aligned}
\end{equation*}
Now, we shall prove that
 $$\int_M a_\ell\, d\nu_g =o\(\eps^{\theta_\ell}\)\ \hbox{for any}\ \ell=2,\dots,k.$$
 Let $\ell$ be fixed and let us split 
 $M=\displaystyle{\cup_{h=0}^\ell} \mathcal A_h$ where we agree that $A_0:=M \setminus B_\xi(r_0)$ and the annuli $\mathcal A_1,\dots,\mathcal A_\ell$ are defined in \eqref{anelli}. Then it is clear that
 $$\int_M a_\ell\, d\nu_g=
\sum _{h=0}^\ell \int_{\mathcal A_h}a_\ell\, d\nu_g.$$
\begin{itemize}
\item[$\diamond$]
First, let us consider the case $h=\ell$. By Lemma \ref{yyl} we get
\begin{equation*}
\begin{aligned}
\int_{\mathcal A_\ell}a_\ell\, d\nu_g&:=-\int_{\mathcal A_\ell}\[ F\(\sum_{j=1}^\ell ( W_j +\phi_{j, \e})\)-F\(\sum_{j=1}^\ell  W_j +\sum_{j=1}^{\ell-1}\phi_{j, \e} \)-f\(\sum_{j=1}^\ell  W_j +\sum_{j=1}^{\ell-1}\phi_{j, \e}\)\phi_{\ell, \eps}\]\, d\nu_g\\
&-\int_{\mathcal A_\ell}\[F\(\sum_{j=1}^\ell  W_j +\sum_{j=1}^{\ell-1}\phi_{j, \e}\)-F\(\sum_{j=1}^\ell  W_j\)- f\(\sum_{j=1}^\ell  W_j\) \sum_{j=1}^{\ell-1}\phi_{j, \e}\]\, d\nu_g\\
&+\int_{\mathcal A_\ell}\[F\(\sum_{j=1}^{\ell-1}( W_j+\phi_{j, \e})\)-F\(\sum_{j=1}^{\ell-1} W_j\)-f\(\sum_{j=1}^{\ell-1} W_j\)\sum_{j=1}^{\ell-1}\phi_{j, \e}\]\, d\nu_g\\
&+\int_{\mathcal A_\ell}\[f\(\sum_{j=1}^\ell( W_j+\phi_{j, \e})\)-f\(\sum_{j=1}^\ell  W_j +\sum_{j=1}^{\ell-1}\phi_{j, \e}\)\]\phi_{\ell, \eps}\, d\nu_g\\
&-\int_{\mathcal A_\ell}\[f\(\sum_{j=1}^\ell  W_j\)-f\(\sum_{j=1}^{\ell-1} W_j\)-f( W_\ell )\]\sum_{j=1}^{\ell-1}\phi_{j, \e}\, d\nu_g\\
&=O\( \underbrace{\int_{\mathcal A_\ell}\phi_{\ell, \eps}^{p+1}\, d\nu_g}_{=O\( \|\phi_{\ell, \eps}\|^{p+1}\)}\) + O\(\underbrace{\int_{\mathcal A_\ell} \(\sum_{j=1}^\ell  W_j +\sum_{j=1}^{\ell-1}\phi_{j, \e}\)^{p-1}\phi_{\ell, \eps}^2\, d\nu_g}_{=\ O\( \|\phi_{\ell, \eps}\|^2\)}\)\\ &
+\underbrace{O\(\int_{\mathcal A_\ell}\(\sum_{j=1}^{\ell-1}\phi_{j, \e}\)^{p+1}\, d\nu_g\)}_{(I)}\\
&+\underbrace{O\(\int_{\mathcal A_\ell}\(\sum_{j=1}^\ell  W_j\)^{p-1}\(\sum_{j=1}^{\ell-1}\phi_{j, \e}\)^2\, d\nu_g\)}_{(II)}\\
&+\underbrace{O\(\int_{\mathcal A_\ell}\(\sum_{j=1}^{\ell-1}  W_j\)^p \sum_{j=1}^{\ell-1}\phi_{j, \e}\, d\nu_g\)}_{(III)} + \underbrace{O\(\int_{\mathcal A_\ell}\(\sum_{j=1}^{\ell-1}  W_j\)  W_\ell ^{p-1}\sum_{j=1}^{\ell-1}\phi_{j, \e}\, d\nu_g\)}_{(IV)}
\\ &=o\(\e^{\theta_\ell}\),
\end{aligned}
\end{equation*}
because of the rate of the error term $\phi_{\ell,\eps}$ given in \eqref{stimaphi1} and the following four new estimates.\\
\begin{itemize}
\item[$\diamond\diamond$]
For any $j=1,\dots,\ell-1,$ by the pointwise estimate of $\phi_{j, \eps}$ in \eqref{stimalinfty} we get
\begin{equation}\label{ok6}
|(I)| \le c\int_{\mathcal A_\ell}|\phi_{j, \e}^{p+1}|\, d\nu_g \le c \mu_j^{-N} \(\textrm{meas}\ \mathcal A_\ell\)^{N} \le c \mu_j ^{-N}(\mu_\ell\mu_{\ell-1})^{\frac N 2} \le c\(\frac{\mu_\ell}{\mu_{\ell-1}}\)^{\frac N 2}.
\end{equation}
\item[$\diamond\diamond$]
If $j=1,\dots,\ell-1$ and $m=1,\dots,\ell-1$ we get by \eqref{ok3} and \eqref{ok6}
\begin{equation*}
\begin{aligned}
|(II)|&\le c\int_{\mathcal A_\ell} | W_j^{p-1}\phi_{m, \eps}^2|\, d\nu_g\le c | W_j|_{{2N\over N-2}, \mathcal A_\ell}^{p-1}|\phi_{m, \eps}|^2_{{2N\over N-2}, \mathcal A_\ell} \\
&\le   c \(\frac{\sqrt{\mu_\ell\mu_{\ell-1}}}{\mu_j}\right)^2 \mu_m^{-N+2}\(\mu_\ell\mu_{\ell-1}\)^{\frac{N-2}{2}}\le c\(\frac{\mu_\ell}{\mu_{\ell-1}}\)^{\frac N 2} ,
\end{aligned}
\end{equation*}
while for $j=\ell$  and $m=1,\dots,\ell-1$, by the pointwise estimate of $\phi_{m, \eps}$ in \eqref{stimalinfty} we get
\begin{equation*}
\begin{aligned}
|(II)|&\le c\int_{\mathcal A_\ell} | W_\ell ^{p-1}\phi_{m, \eps}^2| \, d\nu_g\le c \mu_m^{-N+2}\mu_\ell^2
\int\limits_{\sqrt{\mu_\ell\mu_{\ell+1}}\le|x|\le\sqrt{\mu_{\ell-1}\mu_\ell}} { 1\over|x|^{4}}dx\\
& \le c \mu_m^{-N+2}\mu_\ell^2 (\mu_\ell\mu_{\ell-1})^{\frac{N-4}{2}} \le c\(\frac{\mu_\ell}{\mu_{\ell-1}}\)^{\frac N 2} .
\end{aligned}
\end{equation*}
\item[$\diamond\diamond$]
 If $j=1,\dots,\ell-1$ and $m=1,\dots,\ell-1$ by \eqref{ok3} and \eqref{ok6} we immediately get
\begin{equation*}
\begin{aligned}
|(III)|&\le \int_{\mathcal A_\ell}| W_j^p \phi_{m, \eps}|\, d\nu_g\le c | W_j|_{{2N\over N-2}, \mathcal A_\ell}^{p}|\phi_{m, \eps}|_{{2N\over N-2}, \mathcal A_\ell} \le c\(\frac{\mu_\ell}{\mu_{\ell-1}}\)^{\frac N 2} .
\end{aligned}
\end{equation*} \item[$\diamond\diamond$]
 If $j=1,\dots,\ell-1$ and $m=1,\dots,\ell-1$ by \eqref{ok4} and \eqref{ok6} we immediately get
\begin{equation*}
\begin{aligned}
|(IV)|&\le \int_{\mathcal A_\ell}| W_\ell ^{p-1} W_j \phi_{m, \eps}|\, d\nu_g\le c | |W_\ell| ^{p-1} W_j|_{\frac{2N}{N+2}, \mathcal A_\ell}|\phi_{m, \eps}|_{\frac{2N}{N-2}, \mathcal A_\ell}\le c\(\frac{\mu_\ell}{\mu_{\ell-1}}\)^{\frac N 2} .
\end{aligned}
\end{equation*}
\end{itemize}
\item[$\diamond$]
Now, let us consider the case   $h=0, \ldots, \ell-1$. By Lemma \ref{yyl} we get
\begin{equation}\label{ok1000}
\begin{aligned}
\int_{\mathcal A_h} a_\ell\, d\nu_g &:=-\int_{\mathcal A_h}\[ F\(\sum_{j=1}^\ell ( W_j +\phi_{j, \e})\)-F\(\sum_{j=1}^\ell  W_j +\sum_{j=1}^{\ell-1}\phi_{j, \e} \)-f\(\sum_{j=1}^\ell  W_j +\sum_{j=1}^{\ell-1}\phi_{j, \e}\)\phi_{\ell, \eps}\]\\
&-\int_{\mathcal A_h}\[F\(\sum_{j=1}^\ell  W_j +\sum_{j=1}^{\ell-1}\phi_{j, \e}\)-F\(\sum_{j=1}^{\ell-1}(  W_j+\phi_{j, \e})\)- f\(\sum_{j=1}^{\ell-1} ( W_j+\phi_{j, \e})\) W_\ell \]\\
&+\int_{\mathcal A_h}\[F\(\sum_{j=1}^{\ell} W_j\)-F\(\sum_{j=1}^{\ell-1} W_j\)-f\(\sum_{j=1}^{\ell-1} W_j\) W_\ell \]\\
&+\int_{\mathcal A_h}\[f\(\sum_{j=1}^\ell( W_j+\phi_{j, \e})\)-f\(\sum_{j=1}^\ell  W_j +\sum_{j=1}^{\ell-1}\phi_{j, \e}\)\]\phi_{\ell, \eps}+\int_{\mathcal A_h} f( W_\ell )\sum_{j=1}^{\ell-1}\phi_{j, \e} \\
& -\int_{\mathcal A_h} f\(\sum_{j=1}^{\ell-1}( W_j+\phi_{j, \e})\) W_\ell  +\int_{\mathcal A_h} f\(\sum_{j=1}^{\ell-1}  W_j\)  W_\ell \\
&=O\( \underbrace{\int_{\mathcal A_h}\phi_{\ell, \eps}^{p+1}}_{=\ O\( \|\phi_{\ell, \eps}\|^{p+1}\)}\) + O\(\underbrace{\int_{\mathcal A_h} \(\sum_{j=1}^\ell  W_j +\sum_{j=1}^{\ell-1}\phi_{j, \e}\)^{p-1}\phi_{\ell, \eps}^2}_{=\ O\( \|\phi_{\ell, \eps}\|^2\)}\)\\ &
+O\(\underbrace{\int_{\mathcal A_h} W_\ell ^{p+1}}_{(I')}\)+O\(\underbrace{\int_{\mathcal A_h}\(\sum_{j=1}^{\ell-1}(  W_j+\phi_{j, \e})\)^{p-1} W_\ell ^2}_{(II')}\)\\
&+O\(\underbrace{\int_{\mathcal A_h}  W_\ell ^p\(\sum_{j=1}^{\ell-1} \phi_{j, \e}\)}_{(III')} \)\\
&+O\(\underbrace{-\int_{\mathcal A_h} f\(\sum_{j=1}^{\ell-1}( W_j+\phi_{j, \e})\) W_\ell  +\int_{\mathcal A_h} f\(\sum_{j=1}^{\ell-1}  W_j\)  W_\ell }_{(IV')}\) \\
&=o(\e^{\theta_\ell}),
\end{aligned}
\end{equation}
because of the rate of the error term $\phi_{\ell,\eps}$ given in \eqref{stimaphi1} and the following four new estimates.\\
\begin{itemize}
\item[$\diamond\diamond$] 
If $h=1,\dots,\ell-1$ by \eqref{ok1} we immediately get
$$|(I')|\le   c \(\frac{\mu_\ell}{\mu_{\ell-1}}\)^{\frac N 2}.$$
\item[$\diamond$] 
If $h=1,\dots,\ell-1$ and $j=1,\dots,\ell-1$ by \eqref{ok1}, \eqref{ok2} and \eqref{ok5}  we immediately get
\begin{equation*}\begin{aligned}
|(II')|& \le c \int_{\mathcal A_h }|  W_j^{p-1}W_{\ell}^2|\, d\nu_g +c\int_{\mathcal A_h}|\phi_{j, \e}^{p-1} W_\ell ^2|\, d\nu_g \\
&\le c||W_j|^{p-1}W_{\ell}|_{{2N\over N+2},\mathcal A_h}| W_{\ell}|_{{2N\over N-2},\mathcal A_h}+c
||\phi_{j,\eps}|^{p-1}W_{\ell}|_{{2N\over N+2},\mathcal A_h}| W_{\ell}|_{{2N\over N-2},\mathcal A_h}\\ &
 \le c   \(\(\frac{\mu_\ell}{\mu_{\ell-1}}\)^{\frac{N }{2}}\) .
\end{aligned}
\end{equation*}
\item[$\diamond\diamond$] 
If $h=1,\dots,\ell-1$ and $j=1,\dots,\ell-1$ we have to estimate the term $(III') ,$ by distinguish three cases, namely 
$j\le h-1$, $j=h$ and $j\ge h+1.$\\

If $j\ge h+1,$ by \eqref{ok1}, by \eqref{muj} and \eqref{stimaphi1} we get
\begin{equation*}
\begin{aligned}
\left|\int_{\mathcal A_h} W_\ell^p \phi_{j, \eps}\, d\nu_g\right|& \le  c \| \phi_{j, \eps}\|| W_{\ell}|^p_{{2N\over N-2},\mathcal A_h}
\le c\eps^{\frac p2\theta_j}{\mu_\ell^{N+2\over2}\over\(\mu_h\mu_{h+1}\)^{N+2\over4}}\\
&\le c\eps^{\frac p2\theta_ {h+1}}\eps^{\gamma_\ell {N+2\over2}-(\gamma_h+\gamma_{h+1} ) {N+2\over4}}=O\(\e^{p\theta_\ell}\)
\end{aligned}
\end{equation*}
because  $h\le \ell-1$ and so
 $${\frac p 2 \theta_{h+1}+\frac{N+2}{2}\gamma_\ell -\frac{N+2}{4}\gamma_h-\frac{N+2}{4}\gamma_{h+1}=\frac{N+2}{2}(\gamma_\ell-\gamma_h)}\ge \frac{N+2}{2}(\gamma_\ell-\gamma_{\ell-1}).
 $$ 
If $j\le h-1$, since $\mathcal A_h \subset B_\xi (\mu_j)$ we can use the pointwise estimate \eqref{stimalinfty} for $\phi_{j, \eps}$ and so
\begin{equation}\label{ok110}
\begin{aligned}
\left|\int_{\mathcal A_h} W_\ell^p \phi_{j, \eps}\, d\nu_g\right|& \le c
\mu_j^{-\frac{N-2}{2}}\int_{\mathcal A_h} | W_\ell |^p\, d\nu_g\\ &\le c \mu_j^{-\frac{N-2}{2}}\mu_\ell^{\frac{N-2}{2}}\int\limits_{ {\sqrt{\mu_h\mu_{h+1}}\over \mu_\ell}\le|y|\le {\sqrt{\mu_h\mu_{h-1}}\over \mu_\ell}}\frac{1}{(1+|y|^2)^{\frac{N+2}{2}}}dy\\
& \le  c \mu_{h-1}^{-\frac{N-2}{2}}\mu_\ell^{\frac{N-2}{2}}\frac{\mu_\ell^2}{\mu_h\mu_{h+1}}\\
& \le c \(\frac{\mu_\ell}{\mu_{\ell-1}}\)^{\frac N 2} \(\frac{\mu_{\ell-1}}{\mu_{\ell-2}}\)^{\frac{N-2}{2}}.
\end{aligned}
\end{equation}
 If $j=h $ we split the annulus 
$$\mathcal A_{h}=\underbrace{\{\sqrt{\mu_{h}\mu_{h+1}}\le d_g(x,\xi)\le \mu_{h} \}}_{\mathcal A'_{h}}\cup
\underbrace{\{\mu_{h}  \le d_g(x,\xi)\le  \sqrt{\mu_{h}\mu_{h-1}}\}}_{\mathcal A''_{h}}$$
and we get combining the previous estimates
 \begin{equation*}
 \begin{aligned}
 \int_{\mathcal A_h} W_\ell^p \phi_{h, \eps}\, d\nu_g&=\int_{ { \mathcal A}'_h} W_\ell^p \phi_{h, \eps}\, d\nu_g+\int_{ { \mathcal A}''_h} W_\ell^p \phi_{h, \eps}\, d\nu_g
 \\ &\le c
\mu_h^{-\frac{N-2}{2}}\int_{\mathcal A'_h} | W_\ell |^p\, d\nu_g+ c \| \phi_{h, \eps}\|| W_{\ell}|^p_{{2N\over N-2},\mathcal A''_h}\\
&\le c \mu_h^{-\frac{N-2}{2}}\mu_\ell^{\frac{N-2}{2}}\frac{\mu_\ell^2}{\mu_h\mu_{h+1}}+  c\eps^{\frac p2\theta_h}\({\mu_\ell \over \mu_h} \)^{N+2\over2}\\ &\le c\({\mu_\ell\over\mu_{\ell-1}}\)^{N\over2}.
 \end{aligned}
\end{equation*}
\item[$\diamond\diamond$] 
We need to estimate the last term $(IV').$ We have to distinguish two cases  $h=0, \ldots, \ell-2$ and $h=\ell-1.$
By Lemma \ref{yyl} we deduce that if $h=0, \ldots, \ell-2$ then
\begin{equation}\label{ok2000}
|(IV')|\le c \sum\limits_{j=1}^{\ell-1}\underbrace{ \int_{\mathcal A_h} | W_j^p  W_\ell |\, d\nu_g}_{(i)} +c\sum\limits_{j=1}^{\ell-1}  \underbrace{\int_{\mathcal A_h}||\phi_{j, \e}|^p  W_\ell |\, d\nu_g }_{(ii)}
\end{equation}
while if $h=\ell-1$ we get
$$
|(IV')|\le c \sum_{j,m=1}^{\ell-1}\underbrace{\int_{\mathcal A_{\ell-1}} \left| |W_j|^{p-1} \phi_{m, \e}   W_\ell \right|\, d\nu_g}_{(iii)}  + \sum_{j=1}^{\ell-1}\underbrace{\int_{\mathcal A_{\ell-1}}\left||\phi_{j, \e}|^p  W_\ell \right|\, d\nu_g }_{(ii)}.
$$
\begin{itemize}

\item[$\diamond\diamond\diamond$]  We estimate $(i)$.\\

Let $h=1,\dots,\ell-2.$ If $j=h$
\begin{equation*}
\begin{aligned} \int_{\mathcal A_h} | W_j^p  W_\ell |\, d\nu_g&\le c \int_{B\(\frac{\sqrt{\mu_h\mu_{h-1}}}{\mu_h}\)\setminus B\(\frac{\sqrt{\mu_h\mu_{h+1}}}{\mu_h}\)}\frac{\mu_h^{\frac {N- 2}2}}{(1+|y|^2)^{\frac{N+2}{2}}}\frac{\mu_\ell^{\frac{N-2}{2}}}{(\mu_\ell^2+\mu_h^2|y|^2)^{\frac{N-2}{2}}}\, dy\\ &\le c\(\frac{\mu_\ell}{\mu_h}\)^{\frac{N-2}{2}}=o\(\e^{\theta_\ell}\),
\end{aligned}
\end{equation*}

if $j\ge h+1$
\begin{equation*}
\begin{aligned}
 \int_{\mathcal A_h} | W_j^p  W_\ell |\, d\nu_g&\le \int_{B\(\frac{\sqrt{\mu_h\mu_{h-1}}}{\mu_j}\)\setminus B\(\frac{\sqrt{\mu_h\mu_{h+1}}}{\mu_j}\)}\frac{\mu_j^{\frac N 2 -1}}{(1+|y|^2)^{\frac{N+2}{2}}}\frac{\mu_\ell^{\frac{N-2}{2}}}{(\mu_\ell^2+\mu_j^2|y|^2)^{\frac{N-2}{2}}}\, dy\\
&\le c \mu_j^{-\frac{N-2}{2}}\mu_\ell^{\frac{N-2}{2}}\int_{\frac{\sqrt{\mu_h\mu_{h+1}}}{\mu_j}}^{\frac{\sqrt{\mu_h\mu_{h-1}}}{\mu_j}}\frac{r^{N-1}}{(1+r^2)^{\frac{N+2}{2}}r^{N-2}}\, dr\\
&\le c \frac{\mu_\ell^{\frac{N-2}{2}}\mu_j^{\frac{N+2}{2}}}{(\mu_h\mu_{h+1})^{\frac{N}{2}}}\\ &\le c\(\frac{\mu_\ell}{\mu_{h}}\)^{\frac{ N -2}{2}}\(\frac{\mu_{h+1}}{\mu_{h}}\) = o(\e^{\theta_\ell})
\end{aligned}
\end{equation*}
and if $j\le h-1$
\begin{equation*}
\begin{aligned}
 \int_{\mathcal A_h} | W_j^p  W_\ell |\, d\nu_g&\le \int_{B\(\frac{\sqrt{\mu_h\mu_{h-1}}}{\mu_j}\)\setminus B\(\frac{\sqrt{\mu_h\mu_{h+1}}}{\mu_j}\)}\frac{\mu_j^{\frac N 2 -1}}{(1+|y|^2)^{\frac{N+2}{2}}}\frac{\mu_\ell^{\frac{N-2}{2}}}{(\mu_\ell^2+\mu_j^2|y|^2)^{\frac{N-2}{2}}}\, dy\\
&\le c\mu_j^{-\frac{N-2}{2}}\mu_\ell^{\frac{N-2}{2}}\int_{\frac{\sqrt{\mu_h\mu_{h+1}}}{\mu_j}}^{\frac{\sqrt{\mu_h\mu_{h-1}}}{\mu_j}}\frac{r^{N-1}}{(1+r^2)^{\frac{N+2}{2}}r^{N-2}}\, dr\\
&\le c \mu_j^{-\frac{N-2}{2}}\mu_\ell^{\frac{N-2}{2}}\int_{\frac{\sqrt{\mu_h\mu_{h+1}}}{\mu_j}}^{\frac{\sqrt{\mu_h\mu_{h-1}}}{\mu_j}}r\, dr\\
&\le c \mu_j^{-\frac N 2}\mu_\ell^{\frac{N-2}{2}}\mu_h \mu_{h-1}\le c \(\frac{\mu_\ell}{\mu_{\ell-1}}\)^{\frac{ N -2}{2}}  \(\frac{\mu_{h}}{\mu_{j}}\)\mu_{h-1}.\end{aligned}
\end{equation*}

\item[$\diamond\diamond\diamond$]  We estimate $(ii)$.\\

We have to distinguish some cases. Let $h=1,\dots,\ell-1.$\\

If $j\le h-1$ then by  the pointwise estimate \eqref{stimalinfty} for $\phi_{j, \eps}$ in $\mathcal A_h \subset B_\xi (\mu_j)$   we
get
\begin{equation*}
\begin{aligned}\int_{\mathcal A_{h}}\left||\phi_{j, \e}|^p  W_\ell \right|\, d\nu_g&\le c{\mu_\ell^{N-2\over2}\over \mu_j^{N+2\over2}}\int\limits
_{\sqrt{\mu_h\mu_{h+1}}\le |x|\le \sqrt{\mu_{h-1}\mu_{h }}}{1\over|x|^{N-2}}dx\le c {\mu_\ell^{N-2\over2}\over \mu_j^{N+2\over2}}(\mu_{h-1}\mu_{h })\\ &
\le c\({\mu_\ell\over\mu_{h-1}}\)^{N-2\over2}{\mu_h\over\mu_{h-1}}\le c\({\mu_\ell\over\mu_{\ell-1}}\)^{N-2\over2}\({\mu_{\ell-1}\over\mu_{h-1}}\)^{N-2\over2}.
\end{aligned}
\end{equation*}
If $j\ge h+1$ by \eqref{stimaphi1} and \eqref{ok1} we get
\begin{equation*}
\begin{aligned}\int_{\mathcal A_{h}}\left||\phi_{j, \e}|^p  W_\ell \right|\, d\nu_g&\le \|\phi_{j, \e}\|^p|W_\ell|_{{2N\over N-2},\mathcal A_h}\le c\e^{\frac{p^2}2\theta _j}
{\mu_\ell^{N-2\over2}\over(\mu_h\mu_{h-1})^{N-2\over4}}  \\ & \le
c\e^{\frac{p^2}2\theta _{h+1}}
{\mu_\ell^{N-2\over2}\over(\mu_h\mu_{h-1})^{N-2\over4}}=o\(\e^{\theta_\ell}\)\end{aligned}
\end{equation*}
because by the choice of $\mu_j$ in \eqref{muj} and the definition of $\theta_{h+1}$ in \eqref{tetaelle} we get
$$\frac{p^2}2\theta _{h+1}+\gamma_\ell{N-2\over2}-\(\gamma_h+\gamma_{h+1}\) {N-2\over4}>\(\gamma_\ell-\gamma_{\ell-1}\) {N-2\over2},$$
since $h+1\le j\le \ell-1.$

If $j=h\le\ell-1$   we split the annulus 
$$\mathcal A_{h}=\underbrace{\{\sqrt{\mu_{h}\mu_{h+1}}\le d_g(x,\xi)\le \mu_{h} \}}_{\mathcal A'_{h}}\cup
\underbrace{\{\mu_{h}  \le d_g(x,\xi)\le  \sqrt{\mu_{h}\mu_{h-1}}\}}_{\mathcal A''_{h}}$$
and we get combining the previous estimates
\begin{equation*}
\begin{aligned}\int_{\mathcal A_{h}}\left||\phi_{h, \e}|^p  W_\ell \right|\, d\nu_g&=\int_{\mathcal A'_{h}}\left||\phi_{h, \e}|^p  W_\ell \right|\, d\nu_g+\int_{\mathcal A''_{h}}\left||\phi_{h, \e}|^p  W_\ell \right|\, d\nu_g\\ &\le
 c {\mu_\ell^{N-2\over2}\over \mu_h^{N+2\over2}}\mu_h^2+c\e^{\frac{p^2}2\theta _{h }}
{\mu_\ell^{N-2\over2}\over \mu_h ^{N-2\over2}}\\
&=o\(\e^{\theta_\ell}\)\ \hbox{if}\ h\le \ell-2.
\end{aligned}
\end{equation*}
If $h=\ell-1$ we need to change the estimate of  the term 
\begin{equation*}
\begin{aligned}\int_{\mathcal A'_{\ell-1}}\left||\phi_{\ell-1, \e}|^p  W_\ell \right|\, d\nu_g& \le c{1\over\mu_{\ell-1}^2}\|\phi_{\ell-1, \e}\| |W_\ell|_{{2N\over N+2},\mathcal A'_{\ell-1}}
\\ &\le c{1\over\mu_{\ell-1}^2}\e^{\frac p2\theta_{\ell-1}}\mu_\ell^2\({\mu_\ell\over\mu_{\ell-1}}\)^{N-6\over4}=o\(\e^{\theta_\ell}\)
\end{aligned}
\end{equation*}
because by the definition of $\theta_\ell$ in \eqref{tetaelle} we get
\begin{equation}\label{ok8} \frac p2\theta_\ell+\frac p2\theta_{\ell-1}>\theta_\ell.\end{equation}

\item[$\diamond\diamond\diamond$]  We estimate $(iii)$.\\

If $m\le \ell-2$
\begin{equation*}
\begin{aligned}
  \int_{\mathcal A_{\ell-1}} | W_j^{p-1}\phi_{m, \eps}  W_\ell |\, d\nu_g &\le c \mu_m^{-\frac{N-2}{2}}\int_{ \mathcal A_{\ell-1}} | W_j^{p-1} W_\ell |\, d\nu_g\\
& \le c\mu_m^{-\frac{N-2}{2}}\mu_\ell^{\frac{N-2}{2}}\mu_j^{N-2} \int_{B\(\frac{\sqrt{\mu_{\ell-1}\mu_{\ell-2}}}{\mu_j}\)\setminus B\(\frac{\sqrt{\mu_\ell\mu_{\ell-1}}}{\mu_j}\)}\frac{1}{(1+|y|^2)^2}\frac{1}{(\mu_\ell^2+\mu_j^2|y|^2)^{\frac{N-2}{2}}}\, dy\\
&\le \mu_m^{-\frac{N-2}{2}}\mu_\ell^{\frac{N-2}{2}}\mu_j^{N-2} \int_{B\(\frac{\sqrt{\mu_{\ell-1}\mu_{\ell-2}}}{\mu_j}\)\setminus B\(\frac{\sqrt{\mu_\ell\mu_{\ell-1}}}{\mu_j}\)}\frac{1}{(1+|y|^2)^2}\frac{1}{\mu_j^{N-2}|y|^{N-2}}\, dy\\
& \le c \mu_m^{-\frac{N-2}{2}}\mu_\ell^{\frac{N-2}{2}}\int_{\frac{\sqrt{\mu_\ell\mu_{\ell-1}}}{\mu_j}}^{\frac{\sqrt{\mu_{\ell-2}\mu_{\ell-1}}}{\mu_j}} r\, dr\\
& \le c \mu_m^{-\frac{N-2}{2}}\mu_\ell^{\frac{N-2}{2}}\frac{\mu_{\ell-2}\mu_{\ell-1}}{\mu_j^2}\\
&\le c \(\frac{\mu_\ell}{\mu_{\ell-1}}\)^{\frac{ N -2}{2}}\(\frac{\mu_{\ell-1}}{\mu_{\ell-2}}\)^{\frac{N-4}{2}}.\end{aligned}
\end{equation*}

If $m=\ell-1$   we split the annulus 
$$\mathcal A_{\ell-1}=\underbrace{\{\sqrt{\mu_{\ell }\mu_{\ell-1}}\le d_g(x,\xi)\le \mu_{\ell-1} \}}_{\mathcal A'_{\ell-1}}\cup
\underbrace{\{\mu_{\ell-1}  \le d_g(x,\xi)\le  \sqrt{\mu_{\ell-1}\mu_{\ell-2}}\}}_{\mathcal A''_{\ell-1}}$$
and so
$$ \int_{\mathcal A_{\ell-1}} | W_j^{p-1}\phi_{\ell-1, \eps}  W_\ell |\, d\nu_g= \int_{\mathcal A'_{\ell-1}} | W_j^{p-1}\phi_{\ell-1, \eps}  W_\ell |\, d\nu_g+ \int_{\mathcal A''_{\ell-1}} | W_j^{p-1}\phi_{\ell-1, \eps}  W_\ell |\, d\nu_g.$$

If $j\le\ell-2$ then arguing as before
\begin{equation*}
\begin{aligned}
  \int_{\mathcal A'_{\ell-1}} | W_j^{p-1}\phi_{\ell-1, \eps}  W_\ell |\, d\nu_g   \le c  \mu_{\ell-1}^{-\frac{N-2}{2}}\mu_\ell^{\frac{N-2}{2}}\frac{ \mu^2_{\ell-1}}{\mu_j^2}\le c \(\frac{\mu_\ell}{\mu_{\ell-1}}\)^{\frac{ N -2}{2}}\(\frac{\mu_{\ell-1}}{\mu_{\ell-2}}\)^{2}=o\(\e^{\theta_\ell}\).\end{aligned}
\end{equation*}

If $j=\ell-1$ then 
\begin{equation*}
\begin{aligned}
  \int_{\mathcal A'_{\ell-1}} | W_{\ell-1, \eps}^{p-1}\phi_{\ell-1, \eps}  W_\ell |\, d\nu_g    \le c{1\over\mu_{\ell-1}^2}\e^{\frac p2\theta_{\ell-1}}\mu_\ell^2\({\mu_\ell\over\mu_{\ell-1}}\)^{N-6\over4}=o\(\e^{\theta_\ell}\)
  ,\end{aligned}
\end{equation*}
because of \eqref{ok8}.
If $j\le\ell-1$ by \eqref{ok2} and \eqref{stimaphi1} we get
\begin{equation*}
\begin{aligned}
\int_{\mathcal A''_{\ell-1}} | |W_j|^{p-1}\phi_{\ell-1, \eps}  W_\ell |\, d\nu_g& \le c \|\phi_{\ell-1, \e}\| | |W_j|^{p-1}W_\ell|_{{2N\over N+2},\mathcal A''_{\ell-1}}
\le c \e^{\frac p2\theta_{\ell-1}}  \({\mu_\ell\over\mu_j}\)^{2}\({\mu_\ell\over\mu_{\ell-1}}\)^{N-6\over4}\\ & \le c \e^{\frac p2\theta_{\ell-1}} \({\mu_\ell\over\mu_{\ell-1}}\)^{N+2\over4}=o\(\e^{\theta_\ell}\),
\end{aligned}
\end{equation*}
because of \eqref{ok8}.

\end{itemize}

\end{itemize}

\end{itemize}

That concludes the proof.

{\it Step 2.} We shall write the expansion of $J_\e \(\sum_{j=1}^k  W_j\).$ We   will   split the manifold 
 $M=\displaystyle{\cup_{h=0}^k} \mathcal A_h$ where the annuli $\mathcal A_h$ are defined in \eqref{anelli}.
We have
\begin{equation*}
\begin{aligned}
J_\e \(\sum_{j=1}^k  W_j\)&= \sum_{j=1}^k J_\e( W_j)+\sum_{i,j=1\atop i<j}^k\int_M \(\nabla_g W_i \nabla_g  W_j +(\beta_N \R _g +\e)W_i  W_j \)\\
&-\int_M \[F\(\sum_{j=1}^k  W_j\)-\sum_{j=1}^k F( W_j)\]\\
&= \sum_{j=1}^k J_\e( W_j)+\sum_{i,j=1\atop i<j}^k\int_M \(-\mathcal L_g (W_i)+\e W_i -f(W_i)\)  W_j\, d\nu_g \\
& -\int_M \[F\(\sum_{j=1}^k  W_j\)-\sum_{j=1}^k F( W_j)-\sum_{i,j=1\atop i<j}^k f(W_i)  W_j\]\, d\nu_g \\
&= \sum_{j=1}^k \[J_\e( W_j)+\sum_{ i=1 }^{j-1}\int_M \(-\mathcal L_g (W_i)+\e W_i -f(W_i)\)  W_j\, d\nu_g\]\\
& -\sum_{\ell=2}^k\sum_{h=0}^\ell  \int_{\mathcal A_h}  \[F\(\sum_{j=1}^\ell  W_j\)-F\(\sum_{j=1}^{\ell-1} W_j\)-F( W_\ell )-\sum_{j=1}^{\ell-1}f( W_\ell ) W_j\] \, d\nu_g
\\ &
=\sum_{\ell=1}^k a_\ell\end{aligned}
\end{equation*}
where each $a_\ell$ only depends on $d_1,\dots,d_\ell$ and they are defined as 
\begin{equation*} a_1:=J_\e(W_1)\end{equation*}
and for any $\ell=2,\dots,k$ 
\begin{equation*}\begin{aligned} a_\ell:=&  J_\e( W_\ell)+\sum_{ i=1 }^{\ell-1}\int_M \(-\mathcal L_g (W_i)+\e W_i -f(W_i)\)  W_\ell\, d\nu_g \\ &- \sum_{h=0}^\ell  \int_{\mathcal A_h}  \[F\(\sum_{j=1}^\ell  W_j\)-F\(\sum_{j=1}^{\ell-1} W_j\)-F( W_\ell )-\sum_{j=1}^{\ell-1}f( W_\ell ) W_j\] \, d\nu_g. \end{aligned}
\end{equation*}
We shall prove that
\begin{equation}\label{a1} a_1=\frac{K_N^{-N}}{N} +\eps^2\(-A_N|{\rm Weyl}_g(\xi)|^2_g d_1^4+  B_N  d_1^2\) +  o(\e^2)
\end{equation}
and
\begin{equation}
\label{aelle}
a_\ell=\frac{K_N^{-N}}{N}+\eps^{\theta_\ell}\(-C_N\(\frac{d_\ell}{d_{\ell-1}}\)^{\frac{N-2}{2}} + B_N   d_\ell^2\)   +o\(\eps^{\theta_\ell}\),\quad \hbox{if}\ \ell=2,\dots,k.
\end{equation}
Now  in \cite{ep} it has been proved that 
$$J_\e( W_\ell):= \frac{K_N^{-N}}{N} -A_N|{\rm Weyl}_g(\xi)|^2_g \mu_\ell^4+\e B_N   \mu_\ell^2 +  O(\mu_\ell^5)$$
where $$A_N:= \frac{K_N^{-N}}{24 N (N-4)(N-6)};\qquad B_N:=\frac{2(N-1)K_N^{-N}}{N(N-2)(N-4)}$$ and $K_N:=\sqrt{4 \over N(N-2)\omega_N^{\frac{2}{N}}}$ is the sharp constant for the embedding of $D^{1, 2}(\mathbb R^N)$ into $L^{p+1}(\mathbb R^N)$.
By   \eqref{hhhh} we immediately get for any $i=1,\dots,\ell-1$
\begin{equation*}
\begin{aligned}\left|\int_M \(-\mathcal L_g (W_i)+\e W_i -f(W_i)\)  W_\ell\, d\nu_g\right|&=O\(\int_{B(0,r_0)}{\mu_i^{N-2\over2}\over \(\mu_i^2+|x|^2\)^{N-2\over2}}{\mu_\ell^{N-2\over2}\over \(\mu_\ell^2+|x|^2\)^{N-2\over2}}dx\) \\ & =O\(\int_{B(0,r_0)}{\mu_i^{N-2\over2}\over \(\mu_i^2+|x|^2\)^{N-2\over2}}{\mu_\ell^{N-2\over2}\over  |x| ^{N-2}}dx\) \\ &=O\(\(\frac{\mu_\ell}{\mu_i}\)^{N-2 \over 2}\mu_i^2\)=o\(\e^{\theta_\ell}\).\end{aligned}
\end{equation*}
Finally, it remains to estimate  for $h=1,\dots,\ell$
\begin{equation*}  
\begin{aligned}
 I_h:=\int_{\mathcal A_h}  \[F\(\sum_{j=1}^\ell  W_j\)-F\(\sum_{j=1}^{\ell-1} W_j\)-F( W_\ell )-\sum_{j=1}^{\ell-1}f( W_\ell ) W_j\] \, d\nu_g ,
\end{aligned}
\end{equation*}
 
If $h=\ell$ by Lemma \ref{yyl}
\begin{equation*}
\begin{aligned}
I_\ell &=  \int_{\mathcal A_\ell}\[F\(\sum_{j=1}^\ell  W_j\)-F( W_\ell )-f( W_\ell )\sum_{j=1}^{\ell-1} W_j\]\, d\nu_g-\int_{\mathcal A_\ell}F\(\sum_{j=1}^{\ell-1}  W_j\)\, d\nu_g \\
& =O\( \int_{\mathcal A_\ell}\(\sum_{j=1}^{\ell-1} W_j\)^{p+1}\, d\nu_g\)+O\(\int_{\mathcal A_\ell}\(\sum_{j=1}^{\ell-1} W_j\)^2  W_\ell ^{p-1}\, d\nu_g\) \\
&=O\(\sum_{j=1}^{\ell-1} |W_j|_{{2N\over N-2},\mathcal A_\ell}^{p+1}\)+O\(\sum_{j=1}^{\ell-1} |W_j|_{{2N\over N-2},\mathcal A_\ell} | W_\ell ^{p-1} W_j|_{{2N\over N+2},\mathcal A_\ell} \)
\\ &=o\(\e^{\theta_\ell}\),\ \hbox{because of \eqref{ok3} and \eqref{ok4}}
\end{aligned}
\end{equation*}
If $h=0, \ldots, \ell-1$ we get
\begin{equation*}
\begin{aligned}
I_h&= \underbrace{\int_{\mathcal A_h}\[F\(\sum_{j=1}^\ell  W_j\)-F\(\sum_{j=1}^{\ell-1} W_j\)-f\(\sum_{j=1}^{\ell-1}  W_j\) W_\ell \]\, d\nu_g - \int_{\mathcal A_h}F\(  W_\ell \)\, d\nu_g}_{(i)}\\
&\underbrace{- \int_{\mathcal A_h}\(\sum_{j=1}^{\ell-1} W_j\)f( W_\ell )\, d\nu_g}_{(ii)}+\underbrace{\int_{\mathcal A_h}f\(\sum_{j=1}^{\ell-1} W_j\) W_\ell \, d\nu_g}_{(iii)}.
\end{aligned}
\end{equation*}
Now,  $(i)=o\(\eps^{\theta_\ell}\)$  
  as in $(I')$ and $(II')$ in \eqref{ok1000}, $(ii)=o\(\eps^{\theta_\ell}\)$  as in $(III')$ in \eqref{ok1000} (see \eqref{ok110}) and 
$(iii)=o\(\eps^{\theta_\ell}\)$   when $h=1,\dots, \ell-2$  as in $(IV')$ in \eqref{ok2000}.\\\\
 It only remains to estimate $(iii)$ when $h=\ell-1,$ which  contains the leading term given by the interaction of two consecutive bubbles. Indeed
\begin{equation*}
\int_{\mathcal A_{\ell-1}}f\(\sum_{j=1}^{\ell-1} W_j\) W_\ell \, d\nu_g=\int_{\mathcal A_{\ell-1}}\[f\(\sum_{j=1}^{\ell-1} W_j\)-f(W_{\mu_{\ell-1} })\] W_\ell \, d\nu_g+\int_{\mathcal A_{\ell-1}}f(W_{\mu_{\ell-1}}) W_\ell d\nu_g,
\end{equation*}
where the first term  is estimated   as in $(IV')$ in \eqref{ok2000} when $h=\ell-1$ 
\begin{equation}
\begin{aligned}
&\left|\int_{\mathcal A_{\ell-1}}\[f\(\sum_{j=1}^{\ell-1} W_j\)-f(W_{{\ell-1}})\] W_\ell \, d\nu_g\right|\\ & =O\( \sum_{j=1}^{\ell-2}\int_{\mathcal A_{\ell-1}} |W_{\ell-1}^{p-1} W_j W_\ell |\, d\nu_g\)+ O\(\sum_{j=1}^{\ell-2}\int_{\mathcal A_{\ell-1}}| W_j^p W_\ell |\, d\nu_g\)\\ &  =o \(\e^{\theta_\ell} \).
\end{aligned}
\end{equation}
and  
the second term   is the leading term:
\begin{equation}
\begin{aligned}
 \int_{\mathcal A_{\ell-1}}f(W_{\mu_{\ell-1}}) W_\ell d\nu_g&= \(\frac{\mu_\ell}{\mu_{\ell-1}}\)^{\frac{N-2}{2}}\int_{\mathbb R^N}\frac{ U^p(y)}{|y|^{N-2}}\, dy (1+o(1))\\ &= \(\frac{\mu_\ell}{\mu_{\ell-1}}\)^{\frac{N-2}{2}}\underbrace{\frac{2^{N-1}K_N^{-N}\omega_{N-1}}{N \omega_N}}_{:= C_N} (1+o(1)).\end{aligned}
\end{equation}
That concludes the proof.

	\end{proof}

\end{document}